\documentclass[10pt]{tsoart}

\begin{document}

\title[$L$-homology on ball complexes]{$L$-homology on ball complexes \\ and products}


\author{Spiros Adams-Florou}
\address{School of Maths and Stats \\ University of Glasgow \\ 15 University Gardens \\ Glasgow G12 8QW \\ United Kingdom}
\email{Spiros.Adams-Florou@glasgow.ac.uk}
\thanks{The first author would like to thank the University of Glasgow where most of his work for the present article was done. He would also like to thank the Max Planck Institute for Mathematics in Bonn where the first part of this work was done.}

\author{Tibor Macko}
\address{Mathematical Institute \\ Slovak Academy of Sciences \\ \v Stef\'anikova 49 \\ SK-814 73 Bratislava \\ Slovakia}
\email{macko@mat.savba.sk}
\thanks{The second author would like to thank the University of Bonn where most of his work on the present article was done.}

\maketitle

\begin{center}
\today
\end{center}

\begin{abstract}
We construct homology theories with coefficients in $L$-spectra on the category of ball complexes and we define products in this setting. We also obtain signatures of geometric situations in these homology groups and prove product formulae which we hope will clarify products used in the theory of the total surgery obstruction.
\end{abstract}

\tableofcontents


\section{Introduction} \label{sec:intro}


Let $X$ be an $n$-dimensional geometric Poincar\'e complex, that means a finite $CW$-complex satisfying Poincar\'e duality. A fundamental question in topology of manifolds is to decide whether $X$ is homotopy equivalent to an $n$-dimensional topological manifold. Ranicki developed a systematic general theory for answering this question resulting in the definition of the {\it total surgery obstruction} $s(X) \in \SS_{n} (X) = L_{n-1} (\Lambda^{c}_{\ast} (X))$ which if $n \geq 5$ is zero if and only if the answer is yes, see \cite{Ranicki(1979)}, \cite{Ranicki(1992)}, \cite{Kuehl-Macko-Mole(2012)}. Here $\SS_{n} (X) = L_{n-1} (\Lambda^{c}_{\ast} (X))$ is the $L$-group of the algebraic bordism category of quadratic chain complexes over $X$ which are locally Poincar\'e and globally contractible and satisfy certain connectivity assumptions which we suppress. The theory can also be used in the relative setting when it answers the question of whether two manifold structures $f_{0} \co M_{0} \xra{\simeq} X$ and $f_{1} \co M_{1} \xra{\simeq} X$ are homeomorphic over $X$, that means whether they define the same element in the topological structure set $\sS^{\TOP} (X)$ in the sense of surgery theory. In fact the theory results in a bijection $s \co \sS^{\TOP} (X) \ra \SS_{n+1} (X) = L_{n} (\Lambda^{c}_{\ast} (X))$, see again \cite{Ranicki(1979)}, \cite{Ranicki(1992)}, \cite{Kuehl-Macko-Mole(2012)}.

In the process of developing the theory the whole setup of geometric surgery and algebraic surgery is used. In particular the set of {\it normal invariants} (alias {\it degree one normal maps}), denoted here $\sN^{\TOP} (X)$, is used and a bijection is obtained
\begin{equation} \label{eqn:quad-signature-over-X}
	\qsign_{X} \co \sN^{\TOP} (X) \ra H_{n} (X ; \bL_{\bullet} \langle 1 \rangle) \cong L_{n} (\Lambda_{\ast} (X)).
\end{equation}
Here $H_{n} (X ; \bL_{\bullet} \langle 1 \rangle) \cong L_{n} (\Lambda_{\ast} (X))$ is the homology of $X$ with respect to the $1$-connective cover of the quadratic $L$-theory spectrum $\bL_{\bullet}$ \cite[Chapter 15]{Ranicki(1992)} and as the isomorphism suggests it can be obtained as the $L$-group of the algebraic bordism category of quadratic chain complexes over $X$ which are locally Poincar\'e and again satisfy certain connectivity conditions which we suppress. The map \eqref{eqn:quad-signature-over-X} is called {\it the quadratic signature over} $X$, it is obtained by refining the quadratic construction of \cite{Ranicki-II-(1980)}.

In fact the quadratic signature over $X$ provides us in the case that $X$ is an $n$-dimensional manifold such that $\pi = \pi_{1} (X)$ and $n \geq 5$ with an identification of the geometric and algebraic surgery exact sequences:
\[
\xymatrix@C=7mm{
	\cdots \ar[r] & L_{n+1} (\ZZ\pi) \ar[r] \ar[d]_{=} & \sS^{\TOP} (X) \ar[r] \ar[d]_{\qsign_{X}} & \sN^{\TOP} (X) \ar[r]^(0.55){\qsign_{\pi}} \ar[d]_{\qsign_{X}} & L_{n} (\ZZ\pi) \ar[d]_{=} \\
	\cdots \ar[r] & L_{n+1} (\ZZ\pi) \ar[r] & \SS_{n+1} (X) \ar[r] & H_{n} (X ; \bL_{\bullet} \langle 1 \rangle) \ar[r]_(0.6){\textup{asmb}} & L_{n} (\ZZ\pi) \ar[r] & \SS_{n} (X).
}
\]
Here $\qsign_{\pi}$ is the surgery obstruction map obtained via the above mentioned quadratic construction of \cite{Ranicki-II-(1980)} and $\textup{asmb}$ is the assembly map, see \cite{Ranicki(1992)}, \cite{Kuehl-Macko-Mole(2012)} for more details.

However, there are some deficiencies in the setup from \cite{Ranicki(1979)}, \cite{Ranicki(1992)}, \cite{Kuehl-Macko-Mole(2012)} when we are interested in its behavior with respect to products. Firstly recall that there is a well developed theory for products in $L$-groups of rings in \cite[Section 8]{Ranicki-I-(1980)} and \cite[Section 8]{Ranicki-II-(1980)} which includes product formulae for geometric situations. However, when considering the terms in the sequences above, we observe that the setup relies on choosing a simplicial complex model for $X$ and working with categories of modules and chain complexes over it. In \cite[Appendix B]{Ranicki(1992)} it is described how to use a variant of a simplicial diagonal approximation to obtain certain products. However, from our point of view that description is not sufficient since we could not obtain product formulae for the map \eqref{eqn:quad-signature-over-X}. More similar products appear in \cite[Chapter 21]{Ranicki(1992)} and in the Appendix to \cite{Lueck-Ranicki(1992)} again referring to the simplicial diagonal approximation.

Recently Laures and McClure \cite{Laures-McClure(2009)} used ball complexes instead of simplicial complexes to construct $L$-theory spectra with good multiplicative properties. Moreover, as a byproduct, they constructed a cohomology theory
\begin{equation} \label{eqn:cohlgy-with-L-coefficients}
	X \mapsto H^{n} (X ; \bL_{\bullet}) \cong L_{n} (\Lambda^{\ast} (X))
\end{equation}  
with input the category of ball complexes and where $\Lambda^{\ast} (X)$ is a certain algebraic bordism category (different from $\Lambda_{\ast} (X)$), see our Section \ref{sec:homology-theory} and Theorem 16.1 and Remark 16.2 in \cite{Laures-McClure(2009)} for more details. On the other hand, they did not construct a homology theory and they also did not consider the quadratic signature over $X$ map \eqref{eqn:quad-signature-over-X}. It is our aim in this paper to construct such a homology theory and signature map, to construct products of the shape stated as \eqref{eqn:products-on-L-groups-over-X} and \eqref{eqn:products-on-L-groups-over-X-relative-case}, and to obtain product formulae, stated as \eqref{eqn:product-formula-for-signatures-over-X} and \eqref{eqn:product-formula-for-signatures-over-X-relative-case} below. A special case of one of the product formulae gives that for a closed $k$-dimensional manifold $F$ and a closed $n$-dimensional manifold $X$ we have a commutative diagram
\begin{equation} \label{eqn:products-commute-with-quad-sign-over-X}
	\begin{split}	
		\xymatrix{
		\sN (X) \ar[d]_{\qsign_{X}} \ar[rr]^{\id_{F} \times \blank} & & \sN (X \times F) \ar[d]^{\qsign_{X \times F}} \\
    	H_{n} (X ; \bL_{\bullet} \langle 1 \rangle) \ar[rr]_(0.45){\ssign_{F} (F) \otimes \blank} & & H_{n+k} (X \times F ; \bL_{\bullet} \langle 1 \rangle)).
		}
	\end{split}
\end{equation}
Here $\ssign_{F} (F) \in H_{k} (F ; \bL^{\bullet} \langle 0 \rangle)$ is {\it the symmetric signature over} $F$ of $F$, see \cite{Ranicki(1992)}, \cite{Kuehl-Macko-Mole(2012)}. It refines the symmetric signature over $\pi_{1} (F)$ constructed in \cite{Ranicki-II-(1980)} as  $\ssign_{\pi_{1} (F)} (F) \in L^{k} (\ZZ\pi_{1} (F))$. There exists also a relative version when $F$ is a manifold with boundary. Applying the above diagram in the case $(F,\del F) = (D^{k},S^{k-1})$ yields the suspension isomorphism in the bottom row and hence we obtain a geometric description of such a suspension.


As noted above, the formulae we obtain are related to products mentioned in \cite[Appendix]{Lueck-Ranicki(1992)} and \cite[Chapter 21]{Ranicki(1992)}. In particular Proposition 21.1 in \cite[Chapter 21]{Ranicki(1992)} states multiplicativity for the visible symmetric signature of a Poincar\'e complex. The visible symmetric signature is important since the total surgery obstruction is defined as its boundary in the sense of algebraic surgery. Hence it would shed some light on the multiplicative properties of the total surgery obstruction itself. In principle it should be possible to give an easy proof of the formula from Proposition 21.1 in \cite[Chapter 21]{Ranicki(1992)} in our setup, but for such a proof we would also need to formulate the whole algebraic surgery exact sequence for $X$ a ball complex and prove its main properties such as the identification above. This means in particular reproving the algebraic $\pi-\pi$-theorem of \cite[Chapter 10]{Ranicki(1992)} which we hope to work out in a future work. Nevertheless we hope that already the formulae obtained here will have a direct application in a future work along the lines of \cite[Section 15]{Kuehl-Macko-Mole(2012)} where we aim at simplifying the proof of the main theorem about the total surgery obstruction from \cite{Ranicki(1979)}, \cite{Kuehl-Macko-Mole(2012)}, since the proof uses products on homology and cohomology.	


The present work is structured as follows. In Section 2 we review facts about ball complexes, mainly using \cite{McCrory(1975)}. In Sections 3, 4, and 5 we define additive categories $\AA_{\ast} (X)$ and $\AA^{\ast} (X)$ for a ball complex and we show that they possess chain duality in the sense of \cite[Chapter 1]{Ranicki(1979)}. We find both the proof in \cite[Chapter 5]{Ranicki(1979)} and in \cite{Weiss(1992)} of this fact in the case $X$ is a simplicial complex or a finite $\Delta$-set very dense and hence we give all the details and put the chain duality in a more general context. Besides the product formulae this is meant to be another contribution of our paper. Section 6 contains a brief review of $L$-theory for additive categories with chain duality. Once the chain duality is set up, we observe in Section 7 that the proof from \cite{Weiss(1992)} that the assignment $X \mapsto L_{n} (\Lambda_{\ast} (X))$ defines a homology theory essentially works for ball complexes as well. Finally in Sections 8 and 9 we use the chain level version of the products on $L$-groups of rings from \cite{Ranicki-II-(1980)} and prove the desired product formulae for $L$-homology.


\section{Ball complexes and related concepts} \label{sec:ball-cplxs}


We start with a definition and collect some properties of ball complexes and related concepts. The main sources are \cite{McCrory(1975)}, \cite{Buoncristiano-Rourke-Sanderson(1976)}, \cite{Laures-McClure(2009)}.

\begin{definition} \cite{Buoncristiano-Rourke-Sanderson(1976)} \cite{Laures-McClure(2009)} \label{defn:ball-complex} Let $n$ be a natural number, let $K$ be a finite collection of PL balls in $\RR^{n}$ and write $|K|$ for the union $\bigcup_{\sigma \in K} \sigma$. We say that $K$ is a {\it ball complex} if the interiors of the balls of $K$ are disjoint and the boundary of each ball of $K$ is a union of balls of $K$.
\end{definition}

A ball complex is a regular cell complex in the sense of \cite[Section II.6]{Whitehead(1978)}. However, to distinguish these from cell complexes in the next definition we will refer to them as regular $CW$-complexes. In order to have good dual cell decompositions we need a refinement of this concept. Recall that a {\it polyhedron} $X$ is a topological space with a maximal family of triangulations \cite{McCrory(1975)}. For example, the geometric realisation $|K|$ of a simplicial complex $K$ gives such a polyhedron, so that $K$ is one of the triangulations in the family. A {\it cone} on a topological space $X$ is $\textup{cone} (X) := X \times [0,1]/\sim$ where $(x,1) \sim (y,1)$ for all $x$, $y \in X$. The point $c := [(x,1)]$ is called the {\it cone point}.

\begin{definition} \cite{McCrory(1975)} \label{defn:cone-complex}
	\
	
	A {\it cone complex} $\sC$ on a polyhedron $X$ is a locally finite covering of $X$ by compact subpolyhedra, together with a subpolyhedron $\del \alpha$ for each $\alpha \in \sC$, such that
	\begin{enumerate}
		\item for each $\alpha \in \sC$ we have that $\del \alpha$ is a union of elements in $\sC$
		\item for $\alpha \neq \beta$ the interiors $\mathring{\alpha} = (\alpha \smallsetminus \del \alpha)$ and $\mathring{\beta} = (\beta \smallsetminus \del \beta)$ satisfy $\mathring{\alpha} \cap \mathring{\beta} = \emptyset$
		\item for each $\alpha \in \sC$ there is a PL-homeomorphism $\alpha \cong \textup{cone} (\del \alpha)$ rel $\del \alpha$.
	\end{enumerate}
	
	A {\it cell complex} is a cone complex such that each cone $\alpha \in \sC$ is a ball with $\del \alpha$ its boundary sphere.
	
	A {\it structure} for a cone $\alpha \in \sC$ is a choice of a homeomorphism $f_{\alpha} \co \alpha \ra \textup{cone} (\del \alpha)$. The {\it cone point} of $\alpha$ in this structure is the point $f^{-1}_{\alpha} (c)$.
	
	A {\it structured cone complex} is a cone complex with a choice of structure for each cone.
	
	A {\it structured cell complex} is a cell complex which is structured as a cone complex.
\end{definition}

Dual cell decompositions for ball complexes are described in \cite[page 274]{McCrory(1975)}. They are defined by essentially the same procedure as for simplicial complexes. Our principal motivation for working with structured cell (or ball) complexes rather than simplicial complexes is that structured cell (or ball) complexes behave better with respect to products.

Let $X$ be a ball complex. For each ball $\sigma \in X$ choose a point in its interior. Alternatively assume that $X$ is a structured cell complex, then a choice of such a point is given by taking the cone point for each ball. We obtain the canonical derived subdivision $X'$ which is a simplicial complex with $l$-simplices given by sequences $\sigma_{0} < \cdots < \sigma_{l}$ where $\sigma_{i} \in X$, see \cite[Proposition 2.1]{McCrory(1975)}. The dual cell $D(\sigma,X)$ is a subcomplex of $X'$ which consists of the simplices in $X'$ such that $\sigma \leq \sigma_{0}$. The above construction depends on choices of points if $X$ is just a ball complex, or if $X$ is a structured cell complex there are no choices involved, which is from our point of view the main advantage of structured cell complexes. We also note that the space of these choices is contractible which means that from the homotopy theory point of view there is very little difference between the respective categories.

Let $X$ and $F$ be structured cell complexes and let $\sigma \in X$ be an $n$-ball and let $\tau \in F$ be a $k$-ball. Then we have an $(n+k)$-ball $\sigma \times \tau \in X \times F$ and this gives a structured cell complex structure on $X \times F$. In addition we see from the definitions for dual cells that (see also \cite[Proposition 2.3]{McCrory(1975)})
\begin{equation} \label{eqn:dual-cell-of-product}
	D (\sigma \times \tau,X \times F) = D (\sigma,X) \times D (\tau,F).
\end{equation}

We also need some categorical language. We will denote the category of structured cell complexes and inclusions by $\strcellcat$, the category of ball complexes and inclusions by $\ballcat$, the category of regular $CW$-complexes and inclusions by $\regCWcplxcat$ (remember that these are called regular cell complexes in \cite[Section II.6]{Whitehead(1978)}) so that we have forgetful functors
\[
\strcellcat \ra \ballcat \ra \regCWcplxcat.
\]

We finish with a definition and a proposition which we will need later.

\begin{defn}\label{defn:starsAndBallsInBetween}
Let $X$ be a ball complex.
\begin{enumerate}
 \item Let $\mathrm{st}(\sigma)$ denote the open star of a ball $\sigma$, defined as
 \[ \mathrm{st}(\sigma)= \bigcup_{\sigma\leqslant \tau}\mathring{\tau}.\]
 \item For all inclusions of balls $\rho\leqslant\sigma$, let $[\rho:\sigma]$ denote the union of interiors of all balls containing $\rho$ and contained in $\sigma$:
 \[[\rho:\sigma]= \bigcup_{\rho\leqslant\tau\leqslant\sigma}\mathring{\tau} = \mathrm{st}(\rho)\cap \sigma.\]
\end{enumerate}
\end{defn}

\begin{prop}\label{prop:cellChainsOfRhoColonSigmaCtble}
For all strict inclusions of balls $\rho < \sigma$ in $X$ the cellular chain complex $C_*([\rho:\sigma];\Z)$ and cochain complex $C^{-*}([\rho:\sigma];\Z)$ are chain contractible.
\end{prop}
\begin{proof}
By definition $[\rho:\sigma]=\mathrm{st}(\rho)\cap \sigma$. When $\rho<\sigma$ in a regular cell complex we have that $\partial\sigma \setminus \mathrm{st}(\rho)$ is a closed $(|\sigma|-1)$-ball. Hence
\[C_*([\rho:\sigma];\Z) = C_*(\sigma,\partial\sigma\setminus \mathrm{st}(\rho);\Z) \cong C_*(D^{|\sigma|},D^{|\sigma|-1};\Z)\simeq 0.\]
Similarly for $C^{-*}([\rho:\sigma];\Z)$.
\end{proof}

\section{Chain duality} \label{sec:chain-duality}
In defining $L$-theory of a ring $R$ an involution is required. The involution allows one to convert right $R$-modules into left $R$-modules and hence tensor together two right $R$-modules. The $L$-theory of $R$ depends on the choice of involution. When generalising to $L$-theory of an additive category $\A$ one could also use an involution but instead a weaker structure, that of a chain duality on $\A$, is used instead. A chain duality on $\A$ determines an involution on the derived category of chain complexes in $\A$ and chain homotopy classes of chain maps, allowing for the definition of an $n$-dimensional algebraic Poincar\'{e} complex in $\A$ as a finite chain complex which is chain equivalent to its $n$-dual. This weakening is crucial for local duality in later sections.

Throughout the paper $\A$ denotes an additive category and $\BB(\A)$ denotes the additive category of bounded chain complexes in $\A$ together with chain maps. Let $\iota_\A: \A \to \BB(\A)$ denote inclusion into degree $0$ and let $\mathcal{S}_\mathcal{C}:\mathcal{C}\times\mathcal{C} \to \mathcal{C}\times \mathcal{C}$ denote the functor that switches the two factors for any category $\mathcal{C}$. Let $S^n:\BB(\A)\to \BB(\A)$ and $\Sigma^n:\BB(\A)\to \BB(\A)$ be respectively the unsigned and signed suspension functors as in \cite[pp. $25$-$26$]{Ranicki(1992)}.

The total complex of a double chain complex can be used to extend a contravariant (resp. covariant) functor $T_\A:\A \to \BB(\A)$ to a contravariant (resp. covariant) functor $T_\BB: \BB(\A)\to \BB(\A)$ such that $T_\A = T_\BB\circ \iota_\A$. The contravariant case is explained in detail in \cite[p. $26$]{Ranicki(1992)} and the covariant case is obtained by replacing $C_{-p}$ with $C_p$ in all the formulae for the contravariant case.

\begin{prop}\label{prop:extende}
A natural transformation $e_\A: \sF_\A \Rightarrow \sG_\A: \A\to \BB(\A)$ of contravariant (resp. covariant) functors extends to a natural transformation $e_\BB: \sF_\BB \Rightarrow \sG_\BB: \BB(\A) \to \BB(\A)$ of their extensions.
\end{prop}
\begin{proof}
For contravariant functors setting 
\[e_\BB(C)_n = \sum_{p+q=n}e_\A(C_{-p})_q: \sum_{p+q=n}\sF_\A(C_{-p})_q \to \sum_{p+q=n}\sG_\A(C_{-p})_q\]
defines the required natural transformation and for covariant functors replace all instances of $C_{-p}$ with $C_p$ in the above.
\end{proof}

\begin{defn}\label{defn:chainduality}
A \textit{chain duality} on an additive category $\A$ is a pair $(T_\A,e_\A)$ where
\begin{itemize}
 \item $T_\A$ is a contravariant additive functor $T_\A:\A\to \BB(\A)$,
 \item $e_\A$ is a natural transformation $e_\A: T_\BB\circ T_\A \Rightarrow \iota_\A$ such that, for all $M\in \A$:
 \begin{itemize}
 \item $e_\BB(T_\A(M))\circ T_\BB(e_\A(M)) = \id_{T_\A(M)}: T_\A(M) \to T_\BB^2(T_\A(M)) \to T_\A(M)$,
 \item $e_\A(M):T_\BB(T_\A(M)) \to M$ is a chain equivalence.
 \end{itemize}
\end{itemize}
\end{defn}

The dual $T_\A(M)$ of an object $M$ is a chain complex in $\A$ and the extension $T_\BB$ defines the dual of a chain complex $C\in \BB(\A)$ as $T_\BB(C)$. A chain duality is used to define a tensor product of two objects $M,N\in \A$ over $\A$ as
\begin{equation}\label{eqn:tensorOverA}
M\otimes_\A N = \Hom_\A(T_\A(M),N) 
\end{equation}
which is a priori just a chain complex in the category $\mathrm{Ab}$ of Abelian groups. This generalises to chain complexes $C,D\in \BB(\A)$:
\[ C\otimes_\A D = \Hom_\A(T_\BB(C),D).\]
See \cite[p. $26$]{Ranicki(1992)} for the definition of $\Hom_\A(C,D)$ as the total complex of a double complex.

The tensor product $f\otimes_\A f^\prime:C\otimes_\A C^\prime \to D\otimes_\A D^\prime$ of a pair of morphisms $f:C\to D$, $f^\prime:C^\prime\to D^\prime$ in $\BB(\A)$
sends $\phi\in \Hom_\A(T_\BB(C),C^\prime)$ to 
\begin{equation}\label{eqn:tensorOfChainMaps}
(f\otimes_\A f^\prime)(\phi):=T_\BB(f)^*(f^\prime)_*(\phi) = f^\prime\circ \phi\circ T_\BB(f)\in \Hom_\A(T_\BB(C^\prime),D^\prime). 
\end{equation}
Tensor product over $\A$ thus defines a functor
\[-\otimes_\A-: \BB(\A)\times \BB(\A) \to \BB(\mathrm{Ab}).\]

\begin{ex}\label{ex:ringsMotivateGeneralCase}
For a ring $R$ let $\A(R)$ denote the additive category of f.g. free left $R$-modules. Let $T= \Hom_R(-,R): \A(R) \to \A(R)$. Then there is a natural isomorphism
\[e: T^2 \Rightarrow  \id_{\A(R)}\]
with $e(M) = \mathrm{ev}(M)^{-1}$ where 
\[\begin{array}{rcl}
  \mathrm{ev}(M):M&\stackrel{\cong}{\longrightarrow}& T^2(M) = \Hom_R(\Hom_R(M,R),R)\\
  m&\mapsto& (f\mapsto f(m)).
  \end{array}
\] 
One can easily check that $e(T(M))\circ T(e(M))= \id_{T(M)}$ for all $M\in \A(R)$ and thus that $(T,e)$ is a chain duality on $\A(R)$.

There is also a tensor product of f.g. free left $R$-modules such that 
\begin{equation}\label{eqn:tensorOverRIsHomOne}
M\otimes_R M^\prime \cong \Hom_R(T(M),M^\prime) 
\end{equation}
with switch isomorphisms $T_{M,M^\prime}:M\otimes_R M^\prime \stackrel{\cong}{\longrightarrow} M^\prime\otimes_R M$ corresponding to isomorphisms 
\begin{equation}\label{eqn:switchForR}
\begin{array}{rcl}\Hom_R(T(M),M^\prime) &\stackrel{\cong}{\longrightarrow}& \Hom_R(T(M^\prime),M)\\ \phi &\mapsto& e(M)\circ T(\phi).\end{array} 
\end{equation}
\end{ex}

The previous example is a special case of a chain duality, in the sense that duals are only modules rather than chain complexes. It does however motivate how a chain duality $(T_\A,e_\A)$ on an additive category $\A$ is used to define a tensor product by $(\ref{eqn:tensorOverA})$ together with switch isomorphisms
\[T_{C,D}: C\otimes_\A D \stackrel{\cong}{\longrightarrow} D\otimes_\A C\]such that $T_{D,C} = T_{C,D}^{-1}$, for all $C$,$D\in \BB(\A)$. The idea is to use $T_\A$ and $e_\A$ as in $(\ref{eqn:switchForR})$. This is explained in the following Propositions.

\begin{prop}\label{prop:TInducesChainMapOfHoms}
Let $T_\A: \A\to \BB(\A)$ be a contravariant functor with extension $T_\BB$. Then, for all $C,D\in \BB(\A)$, the extension $T_\BB$ induces a chain map
\[ T_\BB: \Hom_\A(C,D) \to \Hom_\A(T_\BB(D),T_\BB(C)).\]
\end{prop}
\begin{proof}
Since $\Hom_\A(C,D)_n = \sum_{p\in\Z} \Hom_\A(C_p,D_{p+n})$ an $n$-chain $\phi\in \Hom_\A(C,D)_n$ is a direct sum $\phi = \sum_{p\in\Z}\phi_p^{p+n}$ with components $\phi_p^{p+n}:C_p\to D_{p+n}$. Similarly \[\Hom_\A(T_\BB(D),T_\BB(C))_n = \sum_{p-q-r+s=n}\Hom_\A(T_\A(D_p)_q,T_\A(C_r)_s).\] The desired chain map 
\[ T_\BB: \Hom_\A(C,D) \to \Hom_\A(T_\BB(D),T_\BB(C))\]is defined by specifying the components
$T_\BB(\phi)_{p,q}^{r,s}:T_\A(D_p)_q\to T_\A(C_r)_s$ as
\begin{equation}\label{eqn:TBphiComponents1pt5}
 (T_\BB)_n(\phi)_{p,q}^{r,s}= \left\{\begin{array}{cc} (-1)^{nq}(-1)^{\frac{1}{2}n(n-1)}T_\A(\phi_{p-n}^p)_q,&s=q,r=p-n\\0, &\text{otherwise.}\end{array}\right.
\end{equation}
\end{proof}

\begin{rmk}\label{rmk:When2NotionsOfTBCoincide}
Let $\phi:C\to D$ be a morphism of $\BB(\A)$, i.e. a $0$-cycle $\phi\in \Hom_\A(C,D)_0$. Then $(T_\BB)_0(\phi)\in \Hom_\A(T_\BB(D),T_\BB(C))_0$ is precisely the image $T_\BB(\phi)$ of $\phi$ under the extended functor $T_\BB: \BB(\A)\to \BB(\A)$.
\end{rmk}

\begin{prop}\label{prop:TBSendsChainEquivsToChainEquivs}
Let $\phi\in \Hom_\A(C,D)_0$ be a chain equivalence. Then $(T_\BB)_0(\phi)\in \Hom_\A(T_\BB(D),T_\BB(C))_0$ is also a chain equivalence.
\end{prop}
\begin{proof}
Let $\phi\in\Hom_\A(C,D)_0$ have chain homotopy inverse $\psi\in\Hom_\A(D,C)_0$ and let $P\in\Hom_\A(C,C)_1$, $Q\in\Hom_\A(D,D)_1$ be chain homotopies
\[P:\psi\circ \phi \simeq \id_C,\quad Q:\phi\circ\psi\simeq \id_D.\]Then $(T_\BB)_0(\phi)\in\Hom_\A(T_\BB(D),T_\BB(C))_0$ has chain homotopy inverse $(T_\BB)_0(\psi)\in\Hom_\A(T_\BB(C),T_\BB(D))_0$. Also $(T_\BB)_1(P)\in \Hom_\A(T_\BB(C),T_\BB(C))_1$ and $(T_\BB)_1(Q)\in\Hom_\A(T_\BB(D),T_\BB(D))_1$ provide chain homotopies
\[(T_\BB)_1(P): (T_\BB)_0(\psi)\circ (T_\BB)_0(\phi)\simeq \id_{T_\BB(C)},\quad (T_\BB)_1(Q): (T_\BB)_0(\phi)\circ (T_\BB)_0(\psi)\simeq \id_{T_\BB(D)}.\]
\end{proof}

\begin{prop}\label{prop:TBIsNatTrans}
A contravariant functor $T_\A: \A\to \BB(\A)$ with extension $T_\BB$ induces a natural transformation also denoted $T_\BB$:
\[T_\BB: -\otimes_\A- \Rightarrow (-\otimes_\A-)\circ (\id_{\BB(\A)}\times T_\BB^2)\circ\mathcal{S}_{\BB(\A)}:\BB(\A)\times \BB(\A)\to \BB(\mathrm{Ab})\]by
\[T_\BB(C,D)= T_\BB: C\otimes_\A D \to D\otimes T_\BB^2(C)\]where $T_\BB$ is the chain map of Proposition \ref{prop:TInducesChainMapOfHoms} with $T_\BB(C)$ in place of $C$.
\end{prop}
\begin{proof}
Let $f:C\to D$, $f^\prime:C^\prime \to D^\prime$ be morphisms of $\BB(\A)$. Commutativity of the diagram
\begin{displaymath}
 \xymatrix{
C\otimes_\A C^\prime \ar[r]^-{T_\BB} \ar[d]_-{f\otimes_\A f^\prime} & C^\prime\otimes_\A T_\BB^2(C)\ar[d]^-{f^\prime\otimes_\A T_\BB^2(f)}\\
D\otimes_\A D^\prime \ar[r]^-{T_\BB} & D^\prime\otimes_\A T_\BB^2(D)
}
\end{displaymath}
can be checked directly using $(\ref{eqn:TBphiComponents1pt5})$ and the formulae for $f^\prime\otimes_\A T_\BB^2(f)$ and $f\otimes_\A f^\prime$.
\end{proof}

\begin{prop}\label{prop:switchFromTB}
For all $C,D\in \BB(\A)$ let 
\[ T_{C,D}: C\otimes_\A D \to D\otimes_\A C\]be defined as the composition
\[ T_{C,D}= (\id_D\otimes_\A e_\BB(C))\circ T_\BB: C\otimes_\A D \to D\otimes_\A T_\BB^2(C) \to D\otimes_\A C\]
where \[T_\BB: C\otimes_\A D \to D\otimes_\A T_\BB^2(C)\]is the chain map of Proposition \ref{prop:TInducesChainMapOfHoms} with $T_\BB(C)$ in place of $C$. Then $T_{C,D}$ is a chain isomorphism with inverse $T_{D,C}$.
\end{prop}
\begin{proof}
$T_{C,D}$ is defined as the composition of two chain maps and is therefore a chain map. The fact that $T_{D,C}\circ T_{C,D}= \id_{C\otimes_\A D}$ follows from the properties of $(T_\A,e_\A)$ being a chain duality on $\A$. More precisely, this uses naturality of $e_\A$ and that \[e_\BB(T_\A(C_p))\circ T_\BB(e_\A(C_p)) = \id_{T_\A(C_p)}: T_\A(C_p) \to T_\BB^2(T_\A(C_p)) \to T_\A(C_p)\]for all $p\in\Z$.
\end{proof}

\begin{rmk}
Applying remark \ref{rmk:When2NotionsOfTBCoincide} to the identity chain map $\id_{T_\BB(C)}:T_\BB(C)\to T_\BB(C)$ we have that $T_\BB: \Hom_\A(T_\BB(C),T_\BB(C))\to \Hom_\A(T_\BB^2(C),T_\BB^2(C))$ sends $\id_{T_\BB(C)}$ to $T_\BB(\id_{T_\BB(C)}) = \id_{T^2_\BB(C)}.$ Consequently, by definition of $T_{C,T_\BB(C)}$ we have
\begin{align}
T_{C,T_\BB(C)}(\id_{T_\BB(C)}) &= e_\BB(C)\circ T_\BB(\id_{T_\BB(C)})\notag\\
&= e_\BB(C)\circ \id_{T^2_\BB(C)}\notag\\
&= e_\BB(C).\notag
\end{align}
Thus $e_\BB(C)$ can be recovered from the switch isomorphisms.
\end{rmk}


\begin{prop}\label{prop:naturaliso}
A chain duality on $\A$ defines a natural isomorphism
\[ T_{-,-}: -\otimes_\A - \Rightarrow (-\otimes_\A-)\circ \mathcal{S}_{\BB(\A)}:\BB(\A)\times\BB(\A) \to \BB(\mathrm{Ab})\]such that \[(T_{-,-}\circ \id_{\mathcal{S}_{\BB(\A)}})\circ T_{-,-}: -\otimes_\A- \Rightarrow (-\otimes_\A-)\circ \mathcal{S}_{\BB(\A)}^2 = (-\otimes_\A-)\]is the identity natural isomorphism of the functor $-\otimes_\A-$.
\end{prop}
\begin{proof}
For all $C,D\in \BB(\A)$, define $T_{-,-}(C,D)= T_{C,D}$ as in Proposition \ref{prop:switchFromTB}. The fact that $T_{-,-}$ is a natural transformation follows directly from the fact that $e_\BB$ and the $T_\BB$ of Proposition \ref{prop:TBIsNatTrans} are natural transformations. The second part follows from the fact that $T_{D,C}\circ T_{C,D} = \id_{C\otimes_\A D}$.
\end{proof}

We have just seen that a chain duality $(T_\A,e_\A)$ is precisely the additional structure on an additive category $\A$ required so that if we define a tensor product over $\A$ by $-\otimes_\A-= \Hom_\A(T_\BB(-),-)$ there is a natural switch isomorphism $T_{-,-}: -\otimes_\A- \Rightarrow (-\otimes_\A-)\circ \mathcal{S}_{\BB(\A)}$ with the property that $T_{D,C} = T_{C,D}^{-1}$ for all $C,D\in \BB(\A)$. 

Conversely, suppose one starts with a contravariant functor $T_\A:\A\to \BB(\A)$ and a natural switch isomorphism $T_{-,-}: -\otimes_\A- \Rightarrow (-\otimes_\A-)\circ \mathcal{S}_{\BB(\A)}$ with the property that $T_{D,C} = T_{C,D}^{-1}$ for all $C,D\in \BB(\A)$. Then $T_\A$ and the switch isomorphism can be used to define a natural transformation $e_\A:T_\BB\circ T_\A\to \iota_\A$ such that $(T_\A,e_\A)$ satisfies all the properties of being a chain duality, except possibly that $e_\A(M): T_\BB(T_\A(M))\to M$ is a chain equivalence for all $M$:

\begin{prop}\label{prop:eTTeFormulaFromSwitchNatIso}
Let $T_\A:\A\to \BB(\A)$ be a contravariant functor and let \[-\otimes_\A- = \Hom_\A(T_\BB(-),-):\BB(\A)\times \BB(\A)\to \BB(\mathrm{Ab}).\]Suppose there exists a natural isomorphism \[T_{-,-}: -\otimes_\A- \Rightarrow (-\otimes_\A-)\circ \mathcal{S}_{\BB(\A)}\]such that $T_{D,C}\circ T_{C,D} = \id_{C\otimes_\A D}$ for all $C,D\in \BB(\A)$. Let
\[e_\BB(C)= T_{C,T_\BB(C)}(\id_{T_\BB(C)}),\]for all $C\in \BB(\A)$. Then
\begin{enumerate}
 \item this defines a natural transformation\[e_\BB: T_\BB^2 \Rightarrow \id_{\BB(\A)},\]
 \item $T_{C,D}(\phi) = e_\BB(C)\circ T_\BB(\phi)$, for all chain maps $\phi:T_\BB(C)\to D$,
 \item $e_\BB(T_\BB(C))\circ T_\BB(e_\BB(C)) = \id_{T_\BB(C)}$, for all $C\in \BB(\A)$.
\end{enumerate}
\end{prop}
\begin{proof}
\begin{enumerate}
 \item As $T_{-,-}$ is a natural transformation the following diagram commutes for all chain maps $f:C\to D$:
\begin{displaymath}
 \xymatrix@C=2cm{
C \otimes_\A T_\BB (C) \ar[r]^-{T_{C,T_\BB (C)}} \ar[d]_-{f\otimes_\A \id_{T_\BB (C)}} & T_\BB (C)\otimes_\A C \ar[d]^-{\id_{T_\BB (C)}\otimes_\A f}\\
D\otimes_\A T_\BB (C) \ar[r]^-{T_{D,T_\BB (C)}} & T_\BB (C)\otimes_\A D\\
D\otimes_\A T_\BB (D) \ar[u]^-{\id_D\otimes_\A T_\BB (f)} \ar[r]^-{T_{D,T_\BB (D)}} & T_\BB (D)\otimes_\A D \ar[u]_-{T_\BB (f)\otimes_\A \id_D}
}
\end{displaymath}
Observing where $\id_{T_\BB (C)}$ and $\id_{T_\BB (D)}$ map in this diagram we obtain:
\begin{displaymath}
 \xymatrix@C=2cm{
\id_{T_\BB (C)} \ar@{|->}[r] \ar@{|->}[d] & e_\BB(C) \ar@{|->}[d]\\
T_\BB (f) \ar@{|->}[r] & x\\
\id_{T_\BB (D)} \ar@{|->}[u] \ar@{|->}[r] & e_\BB(D) \ar@{|->}[u]
}
\end{displaymath}
where 
\[ T_{D,TC}(T_\BB (f)) = f\circ e_\BB(C) = e_\BB(D)\circ T_\BB^2(f)=x.\]
Hence $e_\BB$ is a natural transformation as required.
 \item Let $\phi:T_\BB(C)\to D$ be a chain map. Then, by naturality of the switch natural isomorphism, the following diagram commutes:
\begin{displaymath}
 \xymatrix@C=2cm{
C\otimes_\A T_\BB(C) \ar[r]^-{T_{C,T_\BB(C)}} \ar[d]_-{\id_C\otimes_\A \phi} & T_\BB(C)\otimes_\A C \ar[d]^-{\phi\otimes_\A \id_C}\\
C\otimes_\A D \ar[r]^-{T_{C,D}} & D\otimes_\A C.
}
\end{displaymath}
Applying this to the chain map $\id_{T_\BB(C)}:T_\BB(C)\to T_\BB(C)$ gives
\[T_{C,D}((\id_C\otimes_\A\phi)(\id_{T_\BB(C)})) = (\phi\otimes_\A \id_C)(T_{C,T_\BB(C)}(\id_{T_\BB(C)})).\]
Thus
\[T_{C,D}(\phi) = T_{C,T_\BB(C)}(\id_{T_\BB(C)})\circ T_\BB(\phi)\]
as required.
\item 
As $T_{C,D}: C\otimes_\A D \to D \otimes_\A C$ is a chain map for all $C,D\in \BB(\A)$ it follows that $(T_{C,D})_0$ sends chain maps to chain maps. Hence $e_\BB(C)= T_{C,T_\BB(C)}(\id_{C\otimes_\A T_\BB(C)})$ is a chain map. Applying part $(1)$ to the equality 
\[ T_{T_\BB(C),C}(T_{C,T_\BB(C)}(\id_{C\otimes_\A T_\BB(C)})) = \id_{C\otimes_\A T_\BB(C)}\]gives
\begin{align}
\id_{C\otimes_\A T_\BB(C)} &= T_{T_\BB(C),C}(T_{C,T_\BB(C)}(\id_{C\otimes_\A T_\BB(C)}))\notag\\
&= T_{T_\BB(C),C}(e_\BB(C))\notag\\
&= e_\BB(T_\BB(C))\circ T_\BB(e_\BB(C))\notag
\end{align}
as required.
\end{enumerate}
\end{proof}

The previous proposition indicates a strategy for defining a chain duality on an additive category $\A$:
\begin{enumerate}
 \item Define a contravariant functor $T_\A: \A\to \BB(\A)$ and obtain its extension $T_\BB:\BB(\A)\to \BB(\A)$. 
 \item Define the tensor product over $\A$ using $T_\BB$ by 
\[ -\otimes_\A- = \Hom_\A(T_\BB(-),-).\]
 \item Exhibit a natural isomorphism \[T_{-,-}: -\otimes_\A- \Rightarrow (-\otimes_\A-)\circ \mathcal{S}_{\BB(\A)}\]with the property that $T_{D,C} = T_{C,D}^{-1}$ for all $C,D\in \BB(\A)$.
 \item Define $e_\BB: T_\BB^2 \to \id_{\BB(\A)}$ by $e_\BB(C)= T_{C,T_\BB(C)}(\id_{C\otimes_\A T_\BB(C)})$.
 \item Prove that $e_\BB(C): T_\BB^2(C) \to C$ is a chain equivalence, for all $C\in \BB(\A)$. 
 \item Using Proposition \ref{prop:eTTeFormulaFromSwitchNatIso} the pair $(T_\BB,e_\BB)$ is a chain duality on $\BB(\A)$ and restricting to $\A\subset \BB(\A)$ this gives a chain duality $(T_\A,e_\A)$ on $\A$.
\end{enumerate}
This strategy is employed in section \ref{sec:chain-duality-on-cats-over-ball-cplxs-saf} to define a chain duality on additive categories over ball complexes. 

\section{Categories over ball complexes} \label{sec:cats-over-ball-cplxs}
In this section we introduce the categories $\A^*[X]$, $\A_*[X]$, $\A^*(X)$ and $\A_*(X)$ and develop some of the tools necessary to define a chain duality on $\A^*(X)$ or $\A_*(X)$. Much of the content can be found in \cite{Ranicki(1992)}; any differences are clearly indicated.

Let $X$ be a ball complex and $\A$ an additive category.

\begin{defn}\label{defn:ballCxAsACategory}
A ball complex $X$ is regarded as a category with objects the set of balls $\sigma\in X$ and morphisms $\tau\to\sigma$ for all inclusions $\tau\leqslant\sigma$.
\end{defn}

\begin{defn}\label{defn:squarebracketcats}
Let $\A^*[X]$ and $\A_*[X]$ denote the additive categories whose objects are respectively covariant and contravariant functors $M:X\to \A$ and whose morphisms are natural transformations of such functors. 
\end{defn}

%

\begin{defn}\label{defn:pushforwardofsquarebracketcategories}
Let $\sF:\A\to \A^\prime$ and $\sG:\A\to\A^\prime$ be covariant and contravariant functors respectively. Then 
post-composition with $\sF$ or $\sG$ defines the following push-forward functors
\begin{align}
\sF_*: \A^*[X] \to (\A^\prime)^*[X],\quad \sF_*: \A_*[X] \to (\A^\prime)_*[X], \notag\\
\sG_*: \A^*[X] \to (\A^\prime)_*[X],\quad \sG_*: \A_*[X] \to (\A^\prime)^*[X]. \notag
\end{align}
\end{defn}

\begin{rmk}\label{rmk:chainCplxsInFunctorCatsBehave}
A bounded chain complex in $\A_*[X]$ is just an object in $\BB(\A)_*[X]$ and similarly for chain maps, so that $\BB(\A_*[X]) = \BB(\A)_*[X]$. Similarly we have $\BB(\A^*[X]) = \BB(\A)^*[X]$.
\end{rmk}

\begin{defn}\label{defn:roundbracketcats}
\begin{enumerate}
 \item An object $M\in \A$ is \textit{$X$-based} if it is expressed as a direct sum \[M = \sum_{\sigma\in X}M(\sigma)\]of objects $M(\sigma)\in \A$. A morphism $f:M\to N$ of $X$-based objects is a collection of morphisms in $\A$ \[f = \{f_{\tau,\sigma}:M(\sigma)\to N(\tau):\sigma,\tau\in X\}.\]
 \item Let the \textit{$X$-graded} category $\mathbb{G}_X(\A)$ be the additive category of $X$-based objects of $\A$ and morphisms $f:M\to N$ of $X$-based objects of $\A$. 
 
The composition of morphisms $f:L\to M$, $g:M\to N$ in $\mathbb{G}_X(\A)$ is the morphism $g\circ f:L\to N$ defined by \[(g\circ f)_{\rho, \sigma} = \sum_{\tau\in X} g_{\rho,\tau}f_{\tau,\sigma}: L(\sigma) \to N(\rho). \]
 \item Let $\brc{\A^*(X)}{\A_*(X)}$ denote the additive category of $X$-based objects $M$ in $\A$ with morphisms $f:M\to N$ such that $f_{\tau,\sigma}:M(\sigma)\to N(\tau)$ is zero unless $\brc{\tau\leqslant \sigma}{\tau\geqslant \sigma.}$
\end{enumerate}
\end{defn}
%
%

\begin{notn}
Let $\A(X)$ denote either $\A^*(X)$ or $\A_*(X)$ and $\A[X]$ either $\A^*[X]$ or $\A_*[X]$ when it doesn't matter which category is used. When this shorthand is used multiple times in a statement it is assumed that the same choice of upper or lower star is consistently made.
\end{notn}

\begin{defn}\label{defn:restrictionToSubBallComplexPair}
Let $X$ be a ball complex and let $Z\subseteq Y \subseteq X$ so that $(Y,Z)$ is a ball complex pair. Define the restriction to $(Y,Z)$
\[ -|_{(Y,Z)}: \A(X) \to \A(Y)\subset \A(X)\]by
\begin{align}
 M|_{(Y,Z)}(\sigma) &= M(\sigma), \forall \sigma\in Y,\; \sigma\notin Z,\notag\\
 M|_{(Y,Z)}(f)_{\tau,\sigma} &= f_{\tau,\sigma}, \forall \tau,\sigma\in Y,\; \tau,\sigma\notin Z.\notag
\end{align}
Denote the restriction to $(Y,\partial Y)$ by $-|_{\mathring{Y}}$ and the restriction to $(Y,\emptyset)$ by $-|_Y$. 
\end{defn}

\begin{defn}\label{defn:totalAssembly}
Define the \textit{total assembly} functor $\mathrm{Ass}: \mathbb{G}_X(\A) \to \A$ by
\begin{align} 
\mathrm{Ass}(M) &= \sum_{\sigma\in X} M(\sigma)\notag\\
\mathrm{Ass}(f:M\to N) &= \{f_{\tau,\sigma}\}_{\tau,\sigma\in X}: \sum_{\sigma\in X}M(\sigma)\to \sum_{\tau\in X}N(\tau).\notag   
\end{align}

Note that total assembly is an equivalence of additive categories. Also denote by $\mathrm{Ass}$ the total assembly restricted to $\A^*(X)$ and $\A_*(X)$.
\end{defn}

\begin{rmk}
Let $\M(\Z)$ denote the additive category of left $\Z$-modules. Following \cite{Ranicki(1992)} we write  
\[\brc{\Z^*(X)= \M(\Z)^*(X)}{\Z_*(X)= \M(\Z)_*(X)}\quad \brc{\Z^*[X]= \M(\Z)^*[X]}{\Z_*[X]= \M(\Z)_*[X].} \]
\end{rmk}

\begin{ex}\label{ex:cellchandcochcxsareXctrld}
Let $X$ be a ball complex and let $\A(\Z)$ be as defined in Example \ref{ex:ringsMotivateGeneralCase} for the ring of integers $\Z$.

The cellular chain complex $C_*(X;\Z)$ of $X$ is naturally a finite chain complex in $\A(\Z)^*(X)$ with $C_*(X;\Z)(\sigma) = S^{|\sigma|}\Z$, for all $\sigma\in X$.

The cellular cochain complex $C^{-*}(X;\Z)$ of $X$ is naturally a finite chain complex in $\A(\Z)_*(X)$ with $C^{-*}(X;\Z)(\sigma) = S^{-|\sigma|}\Z$, for all $\sigma\in X$.
\end{ex}

\begin{ex}\label{ex:subdivsAndDualCells}
Let $X^\prime$ denote the canonical derived subdivision of a structured ball complex $X$. Then the cellular chain and cochain complexes of $X^\prime$ can be viewed as chain complexes in $\BB(\A(\Z)_*(X))$ and $\BB(\A(\Z)^*(X))$ respectively as follows.

Let $D\in \BB(\A(\Z)_*(X))$ be the chain complex with 
\[D(\sigma) = C_*(D(\sigma,X),\partial D(\sigma,X);\Z)\] and differential $(d_D)_{\tau,\sigma}$ obtained by assembling $d_{C_*(X^\prime)}$ appropriately. Then $\mathrm{Ass}(D) = C_*(X^\prime)$.

Similarly the chain complex $D^*\in \BB(\A(\Z)^*(X))$ with 
\[D^*(\sigma) = C^{-*}(D(\sigma,X),\partial D(\sigma,X);\Z)\] and differential $(d_{D^*})_{\tau,\sigma}$ obtained by assembling $d_{C^{-*}(X^\prime)}$ appropriately is such that $\mathrm{Ass}(D^*) = C^{-*}(X^\prime)$.
\end{ex}

Example \ref{ex:subdivsAndDualCells} will be generalised in section \ref{sec:signatures} to a manifold $F$ equipped with a reference map $r_F:F\to L$ to a ball complex. The preimages of dual cells in $L$ are used to dissect $F$.

%
%

\begin{defn}\label{defn:CofXAsCellularChCxs}
Let $C(X)$ denote either the $\Z$-coefficient cellular chain complex $C_*(X;\Z)\in \BB(\A(\Z)^*(X))$ or cochain complex $C^{-*}(X;\Z)\in \BB(\A(\Z)_*(X))$.
\end{defn}

The following Proposition illustrates a very important property of the categories $\A^*(X)$ and $\A_*(X)$. It can be thought of as being analogous to the statement in linear algebra that a triangular matrix is invertible if and only if all its diagonal entries are.

\begin{prop}\label{prop:chcont}
\begin{enumerate}
 \item A chain map $f:C\to D$ in $\A(X)$ is a chain isomorphism if and only if $f_{\sigma,\sigma}:C(\sigma)\to D(\sigma)$ is a chain isomorphism in $\A$ for all $\sigma\in X$.
 \item A chain complex $C$ in $\A(X)$ is chain contractible if and only if $C(\sigma)$ is chain contractible in $\A$ for all $\sigma\in X$.
 \item A chain map $f:C\to D$ of chain complexes in $\A(X)$ is a chain equivalence if and only if $f_{\sigma,\sigma}:C(\sigma)\to D(\sigma)$ is a chain equivalence in $\A$ for all $\sigma\in X$.
\end{enumerate}
\end{prop}
\begin{proof}
These results are well-known for $X$ a simplicial complex (see for example Prop. $4.7$ of \cite{Ranicki(1992)} for parts $(2)$ and $(3)$); Proposition $7.26$ of \cite{adams-florou_homeomorphisms_2012} contains parts $(2)$ and $(3)$, the proof contains the correct formulae but falsely asserts in the last line of page 68 that away from the diagonal $dP=Pd=0$, the correct statement is that $dP+Pd = 0$. The simplicial complex proof generalises verbatim to ball complexes.
\end{proof}

\begin{defn}\label{embeddingfunctors}
Define covariant functors $\mathcal{I}_{X,\A}:\A(X)\to \A[X]$ by sending an object $M\in\A(X)$ to the functor 
$\mathcal{I}_{X,\A}(M):X\to \A$ which sends $\sigma\in X$ to 
\[\sI_{X,\A}(M)(\sigma) = \brcc{\sum_{\rho\leqslant \sigma}M(\rho),}{\A(X) = \A^*(X),}{\sum_{\sigma\leqslant \rho}M(\rho),}{\A(X) = \A_*(X)}\]and a morphism $\tau\to \sigma$ in $X$ to the inclusion map \[\sI_{X,\A}(M)(\tau\to \sigma):\brcc{\sI_{X,\A}(M)(\tau)\hookrightarrow \sI_{X,\A}(M)(\sigma),}{\A(X)=\A^*(X),}{\sI_{X,\A}(M)(\sigma)\hookrightarrow \sI_{X,\A}(M)(\tau),}{\A(X)=\A_*(X).}\]
A morphism $f:M\to N$ in $\A(X)$ is sent to the natural transformation \[\sI_{X,\A}(f):\sI_{X,\A}(M)\Rightarrow \sI_{X,\A}(N):X\to \A\] where \[\sI_{X,\A}(f)(\sigma)= \brcc{\{f_{\tau,\rho}\}_{\rho\leqslant \tau\leqslant \sigma}:\sum_{\rho\leqslant\sigma}M(\rho)\to \sum_{\tau\leqslant \sigma}N(\tau),}{\A(X) = \A^*(X),}{\{f_{\tau,\rho}\}_{\sigma\leqslant \tau\leqslant \rho}:\sum_{\sigma\leqslant\rho}M(\rho)\to \sum_{\sigma\leqslant \tau}N(\tau),}{\A(X) = \A_*(X).}\]
\end{defn}

\begin{rmk}\label{rmk:notationalDifferences}
There are a few notational differences worth highlighting:
\begin{itemize}
 \item For a functor $C:X\to \A$ in $\A[X]$ we use the notation $C(\sigma)$ where Ranicki would write $C[\sigma]$.
 \item In \cite{Ranicki(1992)} the functor $\sI_{X,\A}$ and its extension $(\sI_{X,\A})_\BB$ are both denoted by square brackets, i.e. $\sI_{X,\A}(M)=[M]$, $(\sI_{X,\A})_\BB(C)=[C]$. We prefer to distinguish $\sI_{X,\A}$ from $(\sI_{X,\A})_\BB$.
 \item Combining the above, the notation $\sI_{X,\A}(M)(\sigma)$ corresponds to $[M][\sigma]$ in \cite{Ranicki(1992)}.
\end{itemize}
\end{rmk}

\begin{ex}\label{ex:expressingIasAss}
The functors $\mathcal{I}_{X,\A}$ can be expressed using restriction and assembly functors as follows.
\begin{align}
\mathcal{I}_{X,\A}(M)(\sigma) &= \brc{\mathrm{Ass}(M|_{\sigma}),}{\mathrm{Ass}(M|_{\mathrm{st}(\sigma)}),}\notag \\
\mathcal{I}_{X,\A}(M)(\tau\to \sigma)&=\brc{\mathrm{incl.}:\mathrm{Ass}(M|_{\tau})\hookrightarrow\mathrm{Ass}(M|_{\sigma}),}{\mathrm{incl.}:\mathrm{Ass}(M|_{\mathrm{st}(\sigma)})\hookrightarrow\mathrm{Ass}(M|_{\mathrm{st}(\tau)}),}\notag \\
\mathcal{I}_{X,\A}(f)(\sigma)&= \brc{\mathrm{Ass}(f|_\sigma),}{\mathrm{Ass}(f|_{\mathrm{st}(\sigma)}),}\notag
\end{align}
in the case where \[\A(X) = \brc{\A^*(X)}{\A_*(X)}\]
where the open star $\mathrm{st}(\sigma)$ denotes the simplicial pair $(\mathrm{St}(\sigma),\partial \mathrm{St}(\sigma))$ consisting of the closed star of $\sigma$ and its boundary.
\end{ex}

\begin{ex}\label{ex:whatIsIOfDualCellsEx}
Let $D\in \BB(\A(\Z)_*(X))$ and $D^*\in \BB(\A(\Z)^*(X))$ be as in Example \ref{ex:subdivsAndDualCells}. For a subcomplex $Y\subset X$ let $N(Y,X)$ denote the open neighbourhood of $Y$ in $X$ defined by \[ \mathring{\sigma}\subset N(Y,X)\;\Leftrightarrow\; \sigma\cap Y\neq \emptyset.\]

It follows that
\begin{align}
 \bigcup_{\tau\geqslant\sigma} D(\tau,X)\setminus \partial D(\tau,X) &= D(\sigma,X),\notag\\
 \bigcup_{\tau\leqslant\sigma} D(\tau,X)\setminus \partial D(\tau,X) &= N(\sigma^\prime,X^\prime),\notag
\end{align}
where $X^\prime$ is the canonical derived subdivision of $X$.

Then $(\sI_{X,\A})_\BB(D)$ is the functor in $\BB(\A(\Z)_*[X]) = \BB(\A(\Z))_*[X]$ given by 
\linebreak$(\sI_{X,\A})_\BB(D)(\sigma)= C_*(D(\sigma,X);\Z),$ for all $\sigma\in X$ and such that\[(\sI_{X,\A})_\BB(D)(\tau\to\sigma):C_*(D(\sigma,X);\Z)\to C_*(D(\tau,X);\Z)\]is an inclusion map, for all $\tau\leqslant\sigma$.

Similarly $(\sI_{X,\A})_\BB(D^*)$ is the functor in $\BB(\A(\Z)^*[X]) = \BB(\A(\Z))^*[X]$ given by 
\linebreak$(\sI_{X,\A})_\BB(D^*)(\sigma)= C^{-*}(N(\sigma^\prime,X^\prime);\Z),$ for all $\sigma\in X$ and such that 
\[(\sI_{X,\A})_\BB(D^*)(\tau\to\sigma): C^{-*}(N(\tau^\prime,X^\prime);\Z)\to C^{-*}(N(\sigma^\prime,X^\prime);\Z)\]is an inclusion map, for all $\tau\leqslant\sigma$.
\end{ex}

The tensor product $-\otimes_\A-=\Hom_\A(T_\BB(-),-)$ and its switch natural isomorphism (see Prop. \ref{prop:naturaliso}) together define a tensor product, also denoted $-\otimes_\A-$, on the functor categories $\BB(\A[X])$ as follows.

\begin{defn}\label{defn:tensorprodinfunctorcats}
Let $(T_\A, e_\A)$ be a chain duality on an additive category $\A$ with extensions $(T_\BB,e_\BB)$ as in Proposition \ref{prop:extende}. Define tensor product functors
\begin{align}
 -\otimes_\A- : \BB(\A^*[X]) \times \BB(\A^*[X]) &\to \BB(\Z)^*[X] \notag \\
 -\otimes_\A- : \BB(\A_*[X]) \times \BB(\A_*[X]) &\to \BB(\Z)_*[X] \notag
\end{align}
by sending an object $(C,D)$ of $\BB(\A[X])\times \BB(\A[X])$ to the object $C\otimes_\A D$ of $\BB(\Z)[X]$ given by
\begin{align}
 (C\otimes_\A D)(\sigma) &= C(\sigma)\otimes_\A D(\sigma) = \Hom_\A(T_\BB(C(\sigma)),D(\sigma)),\notag \\
 (C\otimes_\A D)(\tau\to\sigma)&= C(\tau\to\sigma)\otimes_\A D(\tau\to\sigma) = T_\BB(C(\tau\to\sigma))^*(D(\tau\to\sigma))_*\notag
\end{align}
and sending a morphism $(f:C\to D,f^\prime:C^\prime\to D^\prime)$ of $\BB(\A[X])\times \BB(\A[X])$ to the morphism $f\otimes_\A f^\prime$ of $\BB(\Z)[X]$ given by 
\[(f\otimes_\A f^\prime)(\sigma)= T_\BB(f(\sigma))^*f^\prime(\sigma)_*: C(\sigma)\otimes_\A C^\prime(\sigma)\to D(\sigma)\otimes_\A D^\prime(\sigma).\]
\end{defn}

\begin{prop}\label{prop:functorcatdualityisos}
For the tensor product $-\otimes_\A-$ of Definition \ref{defn:tensorprodinfunctorcats}, there are natural isomorphisms 
\[ T_{-,-}:-\otimes_\A- \Rightarrow (-\otimes_\A-)\circ \mathcal{S}_{\BB(\A[X])}\]with $T_{C,D}: C\otimes_\A D \stackrel{\cong}{\longrightarrow} D\otimes_\A C$ the chain isomorphism defined by \[T_{C,D}(\sigma)= T_{C(\sigma),D(\sigma)}: C(\sigma)\otimes_\A D(\sigma) \to D(\sigma)\otimes_\A C(\sigma)\]
where $T_{C(\sigma),D(\sigma)}= (\id_{D(\sigma)}\otimes_\A e_\BB(C(\sigma)))\circ T_\BB$ as in Proposition \ref{prop:switchFromTB}.
\end{prop}
\begin{proof}
This follows directly from Definition \ref{defn:tensorprodinfunctorcats}, naturality of $T_\BB$ (Prop. \ref{prop:TBIsNatTrans}) and naturality of $e_\BB$.
\end{proof}

A chain duality can also be applied componentwise to switch between upper and lower star categories as follows.
\begin{prop}\label{prop:defineMathcalT}
A chain duality $(T_\A,e_\A)$ on $\A$ and its extension $(T_{\BB},e_{\BB})$ to $\BB(\A)$ induce functors
\begin{align}
 \mathcal{T}_\A: \A_*(X) &\to \BB(\A^*(X)),\notag \\
 \mathcal{T}_\A: \A^*(X) &\to \BB(\A_*(X)),\notag
\end{align}
with extensions 
\begin{align}
 \mathcal{T}_{\BB}: \BB(\A_*(X)) &\to \BB(\A^*(X)),\notag \\
 \mathcal{T}_{\BB}: \BB(\A^*(X)) &\to \BB(\A_*(X)),\notag
\end{align}
and equivalences of functors 
\begin{align}
\epsilon_\A: \mathcal{T}_{\BB}\circ \mathcal{T}_\A &\Rightarrow \iota_{\A(X)}:\A(X) \to \BB(\A(X)), \notag \\
\epsilon_\BB: \mathcal{T}_{\BB}\circ \mathcal{T}_{\BB} &\Rightarrow \id_{\BB(\A(X))}:\BB(\A(X)) \to \BB(\A(X)). \notag
\end{align}
\end{prop}
\begin{proof}
For an object $M\in \A(X)$ let $\mathcal{T}_\A(M)$ be the chain complex defined by
\begin{align}
 \mathcal{T}_\A(M)(\sigma)_n &= T_\A(M(\sigma))_{n},\notag\\
 ((d_{\mathcal{T}_\A(M)})_{n})_{\tau,\sigma} &= \brcc{(d_{T_\A(M(\sigma))})_{n},}{\tau=\sigma,}{0,}{\tau\neq\sigma.}\notag
\end{align}
For a morphism $f:M\to N$ in $\A(X)$ let $\mathcal{T}_\A(f):\mathcal{T}_\A(N)\to \mathcal{T}_\A(M)$ be the chain map with $n^\mathrm{th}$ component given by
\[(\mathcal{T}_\A(f)_{n})_{\tau,\sigma}= T_\A(f_{\sigma,\tau})_{n}: T_\A(N(\sigma))_{n}\to T_\A(M(\tau))_{n}.\]The fact that this is a chain map follows directly from functoriality of $T_\A$, as does functoriality of $\mathcal{T}_\A$. 

The extension $\mathcal{T}_\BB = (\mathcal{T}_\A)_\BB$ is defined in the usual way and it follows that $\mathcal{T}_{\BB(\A)}(\mathcal{T}_\A(M))(\sigma) = T_{\BB(\A)}(T_\A(M(\sigma)))$ for all $\sigma\in X$ so the natural transformation
\[\epsilon_\A: \mathcal{T}_\BB\circ \mathcal{T}_\A \Rightarrow \iota_{\A(X)}\]defined by
\[\epsilon_\A(M)_{\tau,\sigma}= \brcc{e_\A(M(\sigma)),}{\tau=\sigma,}{0,}{\tau\neq\sigma}\]is an equivalence of functors. By Proposition \ref{prop:extende} all the above extends to $\BB(\A(X))$.
\end{proof}

\begin{defn}\label{defn:tensoroverZ}
Define tensor product functors
\begin{align}
-\otimes_\Z-: \A(\Z)\times \A &\to \A\notag\\ 
-\otimes_\Z-: \A\times \A(\Z) &\to \A\notag
\end{align}
as follows. For all $\Z^n\in\A(\Z)$, $M\in\A$ let  
\[\Z^n \otimes_\Z M = M\otimes_\Z \Z^n= M^n = \underbrace{M\oplus \ldots \oplus M}_{n\;\mathrm{copies}}.\]
For all morphisms $f:M\to N$ in $\A$ and $g:\Z^m\to \Z^n$ in $\A(\Z)$ let \[g\otimes_\Z f = f\otimes_\Z g: M^m \to N^n\] be the morphism with $(i,j)^\mathrm{th}$ component 
\[(g\otimes_\Z f)_{ij}= (f\otimes_\Z g)_{ij} = \underbrace{f+\ldots + f}_{g_{ij}\;\mathrm{times}}:M\to N.\]
These extend to bounded chain complexes in the usual way:
\begin{align}
 (C\otimes_{\Z} D)_n &= \sum_{p+q=n} C_p\otimes_\Z D_q,\notag\\
 (d_{C\otimes_{\Z} D})_n &= \sum_{p+q=n} (d_C)_p\otimes_\Z (\id_D)_q + (-1)^p(\id_C)_p\otimes_\Z (d_D)_q,\notag\\
 (f\otimes_{\Z} g)_n &= \sum_{p+q=n} f_p\otimes_\Z g_q.\notag
\end{align}
\end{defn}

The following Proposition defines a new tensor product between round and square bracket categories. This tensor product gives rise to a simple interpretation of Ranicki's functor $T_{\A(X)}$ (c.f. \cite[Prop. $5.1$]{Ranicki(1992)}) for $X$ a finite simplicial complex. Extending this to a ball complex $X$, this tensor product is also instrumental in proving that $(T_{\A(X)}, e_{\A(X)})$ is a chain duality.

\begin{prop}\label{prop:defineXgradedtensorproducts}
Let $\A$ be an additive category and $X$ a ball complex. Let $\mathcal{C}$ be either $\A$ or $\A(\Z)$ and $\mathcal{D}$ the other of the two. Then there are tensor product functors
\begin{align}
 -\otimes - : \BB(\mathcal{C}^*(X))\times \BB(\mathcal{D})_*[X] &\to \BB(\A^*(X))\notag\\
 -\otimes - : \BB(\mathcal{C}_*(X))\times \BB(\mathcal{D})^*[X] &\to \BB(\A_*(X))\notag
\end{align}
defined as follows.

For a pair of objects $(C,D)\in \brc{\BB(\mathcal{C}^*(X))\times \BB(\mathcal{D})_*[X]}{\BB(\mathcal{C}_*(X))\times \BB(\mathcal{D})^*[X]}$ let $C\otimes D$ be the object in 
$\brc{\BB(\A^*(X))}{\BB(\A_*(X))}$ with $n$-chains given by
\[ (C\otimes D)_n(\sigma) = \sum_{p+q=n} C_p(\sigma)\otimes_\Z D(\sigma)_q\]for $-\otimes_\Z-$ the tensor product of Definition \ref{defn:tensoroverZ} and differential 
\[ ((d_{C\otimes D})_{n})_{\tau,\sigma}: (C\otimes D)_n(\sigma)\to (C\otimes D)_{n-1}(\tau)\]given by
\begin{equation}\label{eqn:tensordifferential}
\brc{\sum_{p+q=n}\left( ((d_C)_{p})_{\tau,\sigma}\otimes_\Z D(\tau \to \sigma)_q + (-1)^p(\id_{C_p})_{\tau,\sigma}\otimes_\Z (d_{D(\sigma)})_q\right),}{\sum_{p+q=n}\left( ((d_C)_{p})_{\tau,\sigma}\otimes_\Z D(\sigma \to \tau)_q + (-1)^p(\id_{C_p})_{\tau,\sigma}\otimes_\Z (d_{D(\sigma)})_q\right).} 
\end{equation}
For a pair $(f:C\to D, f^\prime: C^\prime\Rightarrow D^\prime)$ of morphisms of \[\brc{\BB(\mathcal{C}^*(X))\times \BB(\mathcal{D})_*[X]}{\BB(\mathcal{C}_*(X))\times \BB(\mathcal{D})^*[X]}\]
the morphism $f\otimes f^\prime: C\otimes C^\prime \to D\otimes D^\prime$ of $\brc{\BB(\A^*(X))}{\BB(\A_*(X))}$ is defined by 
\[((f\otimes f^\prime)_{n})_{\tau,\sigma}= \delta_{\tau\sigma}\sum_{p+q=n}(f_{p})_{\sigma,\sigma}\otimes_\Z f^\prime(\sigma)_q: (C\otimes C^\prime)_n(\sigma)\to (D\otimes D^\prime)_n(\tau)\]where $\delta_{\tau\sigma}$ is a Kronecker delta function.
\end{prop}
\begin{proof}
Straightforward.
\end{proof}

\begin{rmk}\label{rmk:actuallyEightTensorProducts}
One can of course also define tensor products as in Proposition \ref{prop:defineXgradedtensorproducts} but with the order of categories in the domain switched, i.e. tensor products 
\begin{align}
 -\otimes - : \BB(\mathcal{D})_*[X] \times \BB(\mathcal{C}^*(X)) &\to \BB(\A^*(X))\notag\\
 -\otimes - : \BB(\mathcal{D})^*[X]\times \BB(\mathcal{C}_*(X))  &\to \BB(\A_*(X)).\notag
\end{align}
There are natural switch isomorphisms between these tensor products and those of Proposition \ref{prop:defineXgradedtensorproducts}. The chain isomorphisms $C\otimes D \cong D\otimes C$ use a sign of $(-1)^{pq}$ for the component sending $C_p\otimes_\Z D_q$ to $D_q\otimes_\Z C_p$.
\end{rmk}


\begin{ex}\label{ex:tensorWithUnderlineZIsId}
As in Example $4.5$ of \cite{Ranicki(1992)} let $\underline{\Z}\in\brc{\A(\Z)^*[X]\subset \BB(\A(\Z))^*[X]}{\A(\Z)_*[X]\subset \BB(\A(\Z))_*[X]}$ denote the functor with $\underline{\Z}(\sigma) = \Z$ for all $\sigma\in X$ and $\underline{\Z}(\tau\to\sigma) = \id_\Z$ for all $\tau\leqslant \sigma$. 

Then $-\otimes\underline{\Z}$ is seen to be the identity functor on $\brc{\BB(\A_*(X))}{\BB(\A^*(X))}$ by examining the formulae of Proposition \ref{prop:defineXgradedtensorproducts}. As $\underline{\Z}(\sigma)_q$ is $\Z$ when $q=0$ and $0$ otherwise it follows that 
\[ (C\otimes \underline{\Z})_n(\sigma) = \sum_{p+q=n}C_p(\sigma)\otimes_\Z \underline{\Z}(\sigma)_q = C_n(\sigma)\otimes_\Z \Z = C_n(\sigma).\]As $(d_{\underline{\Z}(\sigma)}) = 0$ for all $\sigma\in X$ it follows that the second term in $(\ref{eqn:tensordifferential})$ is always zero. As $D(\tau\to\sigma)_q$ is $id_\Z$ when $q=0$ and $\brc{\tau\leqslant\sigma}{\sigma\leqslant\tau}$ and $0$ otherwise it follows that the first term agrees with the differential of the original chain complex $C$. 
\end{ex}

Another important chain complex to tensor with is $C(X)$, the $\Z$-coefficient cellular chain or cochain complex of $X$. 

\begin{defn}\label{defn:shiftfunctors}
Define the \textit{shift functors}
\begin{align}
 \mathrm{sh}: \BB(\A)^*[X]&\to \BB(\A_*(X))\notag\\
 \mathrm{sh}: \BB(\A)_*[X]&\to \BB(\A^*(X))\notag
\end{align}
as the functors that tensor on the left with the chain complex $C(X)$ of Definition \ref{defn:CofXAsCellularChCxs} using the tensor products of Proposition \ref{prop:defineXgradedtensorproducts}, i.e. $\mathrm{sh} = C(X)\otimes-$.
\end{defn}

\begin{ex}\label{ex:shiftFunctorsAreRanickiCovAssFunctors}
We see that 
\[\mathrm{sh}(C)_n(\sigma) = \brc{C(\sigma)_{n+|\sigma|},}{C(\sigma)_{n-|\sigma|},}\]
with differential
\[((d_{\mathrm{sh}(C)})_{n})_{\tau,\sigma} = \brc{\begin{array}{l} ((d_{C^{-*}(X;\Z)})_{-|\sigma|})_{\tau,\sigma}\otimes_\Z C(\sigma\to\tau)_{n+|\sigma|} \\ \quad + (-1)^{-|\sigma|} ((\id_{C^{-*}(X;\Z)})_{-|\sigma|})_{\tau,\sigma}\otimes_\Z (d_{C(\sigma)})_{n+|\sigma|},\end{array}}{\begin{array}{l} ((d_{C_*(X;\Z)})_{|\sigma|})_{\tau,\sigma}\otimes_\Z C(\tau\to\sigma)_{n-|\sigma|} \\ \quad + (-1)^{|\sigma|} ((\id_{C_*(X;\Z)})_{|\sigma|})_{\tau,\sigma}\otimes_\Z (d_{C(\sigma)})_{n-|\sigma|},\end{array}}\]
for $C\in \brc{\BB(\A)^*[X],}{\BB(\A)_*[X].}$

Morphisms are shifted similarly:
\[ (\mathrm{sh}(f)_{n})_{\tau,\sigma} = \brcc{((\id_{C^{-*}(X;\Z)})_{-|\sigma|})_{\tau,\sigma}\otimes_\Z f(\sigma)_{n+|\sigma|},}{f\in \mathrm{Mor}(\BB(\A)^*[X]),}{((\id_{C_*(X;\Z)})_{|\sigma|})_{\tau,\sigma}\otimes_\Z f(\sigma)_{n-|\sigma|},}{f\in \mathrm{Mor}(\BB(\A)^*[X]).}\]
Comparing these formulae to Definition $4.4$ of \cite{Ranicki(1992)} we see that the shift functors are precisely the same as Ranicki's covariant assembly functors in the case that $X$ is a simplicial complex.\footnote{Modulo possibly a different sign convention for the differential of the total complex of a double complex.} Combining this with Example \ref{ex:tensorWithUnderlineZIsId} recovers the statement in \cite[Example 4.5]{Ranicki(1992)} that the covariant assembly of $\underline{\Z}$ is $C(X)$
\end{ex}

\section{Chain duality on categories over ball complexes} \label{sec:chain-duality-on-cats-over-ball-cplxs-saf}
Let $X$ be a ball complex, $\A$ an additive category and $(T_\A,e_\A)$ a chain duality on $\A$. We follow the approach in section \ref{sec:chain-duality} to define a chain duality on $\A_*(X)$ and $\A^*(X)$ using $(T_\A,e_\A)$.

The following definition can be seen to agree with that of Proposition $5.1$ of \cite{Ranicki(1992)} in the case that $X$ is a simplicial complex.
\begin{defn}\label{defn:XctrldT}
Let $(T_\A,e_\A)$ be a chain duality on an additive category $\A$. Define the contravariant functor $T_{\A(X)}: \A(X)\to \BB(\A(X))$ to be the composition
\[ T_{\A(X)}= \mathrm{sh}\circ (T_\A)_* \circ \sI_{X,\A}\]where $(T_\A)_*$ is the push-forward functor as defined in Definition \ref{defn:pushforwardofsquarebracketcategories}.
\end{defn}

\begin{ex}\label{ex:TAXCWhenCSupportedOnSigma}
Let $X$ be a ball complex and let $\rho\in X$ be fixed. Suppose $C\in \brc{\BB(\A^*(X))}{\BB(\A_*(X))}$ is such that $C(\tau)\neq 0$ if and only if $\tau=\rho$. Then unpacking Definition \ref{defn:XctrldT} and using Example \ref{ex:whatIsIOfDualCellsEx} it follows that
\begin{align}
 T_{\BB(\A(X))}(C)(\sigma) &= \brc{C_*(\sigma,\partial\sigma;\Z)\otimes_\Z T_\BB((I_{X,\A})_\BB(C(\sigma)))}{C^{-*}(\sigma,\partial\sigma;\Z)\otimes_\Z T_\BB((I_{X,\A})_\BB(C(\sigma)))}\notag\\
&= \brc{\Sigma^{|\sigma|}T_\BB(C|_\sigma)}{\Sigma^{-|\sigma|}T_\BB(C|_{\mathrm{st}(\sigma)})}\notag\\
&= \brcc{\Sigma^{|\sigma|}T_\BB(C(\rho)),}{\mathrm{if}\;\sigma\geqslant\rho,}{\Sigma^{-|\sigma|}T_\BB(C(\rho)),}{\mathrm{if}\;\sigma\leqslant\rho,}\notag
\end{align}
and $0$ otherwise. 
\end{ex}

\begin{ex}\label{ex:WhatAreTDAndTDStar}
For the ring of integers $R=\Z$ let $(T,e)$ be the chain duality on $\A(\Z)$ defined in Example \ref{ex:ringsMotivateGeneralCase}. Let $D\in \BB(\A(\Z)_*(X))$ and $D^*\in \BB(\A(\Z)^*(X))$ be as in Examples \ref{ex:subdivsAndDualCells} and \ref{ex:whatIsIOfDualCellsEx}.

By Example \ref{ex:whatIsIOfDualCellsEx} we have that 
\begin{align}
(\sI_{X,\A})_\BB(D)(\sigma)&= C_*(D(\sigma,X);\Z),\notag\\
(\sI_{X,\A})_\BB(D^*)(\sigma)&= C^{-*}(N(\sigma^\prime,X^\prime);\Z)\notag
\end{align}
for all $\sigma\in X$. Therefore it follows that 
\begin{align}
T_{\BB(\A(\Z)_*(X))}(D)(\sigma) &= \Sigma^{-|\sigma|}T(C_*(D(\sigma,X);\Z))\notag\\
&= C^{-|\sigma|-*}(D(\sigma,X);\Z).\notag
\end{align}
Similarly, we have 
\begin{align}
T_{\BB(\A(\Z)^*(X))}(D^*)(\sigma) &= \Sigma^{|\sigma|}T(C^{-*}(N(\sigma^\prime, X^\prime);\Z))\notag\\
&= \Sigma^{|\sigma|}T^2(C_*(N(\sigma^\prime, X^\prime);\Z))\notag\\
&\cong C_{-|\sigma|+*}(N(\sigma^\prime, X^\prime);\Z).\notag
\end{align}
The last chain isomorphism is provided by $\Sigma^{|\sigma|}e(C_*(N(\sigma^\prime, X^\prime);\Z))$.
\end{ex}

In $(\ref{eqn:tensorOverA})$ a tensor product is defined for an additive category $\A$ with chain duality $(T_\A,e_\A)$ where, for objects $M,N\in \A$, $M\otimes_\A N$ is just a chain complex of Abelian groups. This definition only requires the contravariant functor part of the chain duality. Applying this to the categories $\A(X)$ and the contravariant functors $T_{\A(X)}$ of Definition \ref{defn:XctrldT} gives a tensor product over $\A(X)$. In the following it will be necessary to use a refinement of this tensor product where the tensor product of two objects $M,N\in \A(X)$ is not just a chain complex of Abelian groups but rather is fragmented over the ball complex $X$ in the sense of the previous section.   

\begin{defn}\label{defn:tensoroverAofX}
The contravariant functors $T_{\A(X)}:\A(X)\to \BB(\A(X))$ and their extensions $T_{\BB(\A(X))}=(T_{\A(X)})_\BB:\BB(\A(X))\to \BB(\A(X))$ are used to define tensor product functors
\begin{align}
 -\otimes_{\A^*(X)}-= \Hom_{\A^*(X)}(T_{\BB^*(\A(X))}(-),-):\BB(\A^*(X))\times\BB(\A^*(X)) &\to \BB(\Z_*(X)),\notag\\
 -\otimes_{\A_*(X)}-= \Hom_{\A_*(X)}(T_{\BB_*(\A(X))}(-),-):\BB(\A_*(X))\times\BB(\A_*(X)) &\to \BB(\Z^*(X)),\notag
\end{align}
by
\begin{align}
 (C\otimes_{\A^*(X)}D)_n(\sigma) &= \Hom_{\A}(T_{\BB^*(\A(X))}(C)(\sigma),(\sI_{X,\A})_\BB(D)(\sigma))_n,\notag \\
 (d_{C\otimes_{\A^*(X)}D})_{n,\tau,\sigma} &= (-1)^{n-1}(d_{T_{\BB^*(\A(X))}(C)})_{\sigma,\tau}^*((\sI_{X,\A})_\BB(D)(\tau\to\sigma))_*\notag\\
 &\quad + (\id_{T_{\BB^*(\A(X))}(C)})_{\sigma,\tau}^*((\sI_{X,\A})_\BB(d_D)(\sigma))_*,\notag\\
 (f\otimes_{\A^*(X)}f^\prime)_{n,\tau,\sigma} &= T_{\BB^*(\A(X))}(f)_{\sigma,\tau}^*((\sI_{X,\A})_\BB(f^\prime)(\sigma))_*\notag
\end{align}
and
\begin{align}
 (C\otimes_{\A_*(X)}D)_n(\sigma) &= \Hom_{\A}(T_{\BB_*(\A(X))}(C)(\sigma),(\sI_{X,\A})_\BB(D)(\sigma))_n,\notag \\
 (d_{C\otimes_{\A_*(X)}D})_{n,\tau,\sigma} &= (-1)^{n-1}(d_{T_{\BB_*(\A(X))}(C)})_{\sigma,\tau}^*((\sI_{X,\A})_\BB(D)(\sigma\to\tau))_*\notag\\
 &\quad + (\id_{T_{\BB_*(\A(X))}(C)})_{\sigma,\tau}^*((\sI_{X,\A})_\BB(d_D)(\sigma))_*,\notag\\
 (f\otimes_{\A_*(X)}f^\prime)_{n,\tau,\sigma} &= T_{\BB_*(\A(X))}(f)_{\sigma,\tau}^*((\sI_{X,\A})_\BB(f^\prime)(\sigma))_*\notag
\end{align}
\end{defn}

\begin{rmk}
The formulae of Definition \ref{defn:tensoroverAofX} are precisely what you get by replacing $\A$ with $\A(X)$ in $(\ref{eqn:tensorOverA})$ and grouping things accordingly, i.e. the tensor products $-\otimes_{\A(X)}-$ of Definition \ref{defn:tensoroverAofX} are refinements which assemble to give the tensor products of $(\ref{eqn:tensorOverA})$. 
\end{rmk}

The following Proposition applied to objects is stated in the proof of \cite[Proposition $5.1$]{Ranicki(1992)}. Due to different sign conventions this is stated as an equality in \cite{Ranicki(1992)}. This Proposition also indicates the reason for wanting to work with the refined tensor products of Definition \ref{defn:tensoroverAofX}.
\begin{prop}\label{prop:identifytensorprodfunctors}
There are natural isomorphisms of functors
\begin{align}
-\otimes_{\A^*(X)}- &\Rightarrow \mathrm{sh}\circ(-\otimes_\A-)\circ ((\sI_{X,\A})_\BB\times(\sI_{X,\A})_\BB)\notag\\ 
-\otimes_{\A_*(X)}- &\Rightarrow \mathrm{sh}\circ(-\otimes_\A-)\circ ((\sI_{X,\A})_\BB\times(\sI_{X,\A})_\BB).\notag 
\end{align}
\end{prop}
\begin{proof}
It can be shown that there are isomorphisms
\[\Phi_{C,D}:(C\otimes_{\A(X)} D)_n \cong (\mathrm{sh}((\sI_{X,\A})_\BB(C)\otimes_\A (\sI_{X,\A})_\BB(D)))_n\]
for all $C,D\in \BB(\A(X))$. The differentials of source and target agree up to a sign which can be compensated for by choosing the correct signs for the components of $\Phi_{C,D}$. Modulo this isomorphism we have equality
\[ f\otimes_{\A(X)} f^\prime = \mathrm{sh}((\sI_{X,\A})_\BB(f)\otimes_\A(\sI_{X,\A})_\BB(f^\prime))\]for all morphisms $f,f^\prime$ in $\BB(\A(X))$. This proves naturality.
\end{proof}

\begin{prop}\label{prop:natIsoForSimpCats}
There is a natural isomorphism 
\[ T_{-,-}:-\otimes_{\A(X)}- \Rightarrow (-\otimes_{\A(X)}-)\circ \mathcal{S}_{\BB(\A(X))}\]with $T_{C,D}:C\otimes_{\A(X)}D\stackrel{\cong}{\longrightarrow} D\otimes_{\A(X)} C$ defined as the composition
\[T_{C,D}= \Phi_{D,C}^{-1}\circ \mathrm{sh}(T_{(\sI_{X,\A})_\BB(C),(\sI_{X,\A})_\BB(D)})\circ \Phi_{C,D}\]
where $\Phi_{C,D}$ and $\Phi_{D,C}^{-1}$ are isomorphisms as in the proof of Proposition \ref{prop:identifytensorprodfunctors} and $T_{(\sI_{X,\A})_\BB(C),(\sI_{X,\A})_\BB(D)}$ is as defined in Proposition \ref{prop:functorcatdualityisos}. 

As $T_{(\sI_{X,\A})_\BB(D),(\sI_{X,\A})_\BB(C)}=T_{(\sI_{X,\A})_\BB(C),(\sI_{X,\A})_\BB(D)}^{-1}$ for all $C,D\in \BB(\A(X))$ it follows that $T_{D,C}=T_{C,D}^{-1}$ for all $C,D\in \BB(\A(X))$.
\end{prop}

\begin{proof}
For all $C,D\in \BB(\A(X))$, \[T_{C,D}:C\otimes_{\A(X)}D\stackrel{\cong}{\longrightarrow} D\otimes_{\A(X)} C\] is a chain isomorphism as it is defined as the composition of three chain isomorphisms. The fact that $T_{-,-}$ is a natural transformation follows from having natural transformations in Propositions \ref{prop:functorcatdualityisos} and \ref{prop:identifytensorprodfunctors}. 
\end{proof}

\begin{prop}\label{prop:defnOfeAX}
Let $e_{\BB(\A(X))}(C)= T_{C,T_{\BB(\A(X))}(C)}(\id_{T_{\BB(\A(X))}(C)})$, for all $C\in \BB(\A(X))$. Then
\begin{enumerate}
 \item this defines a natural transformation\[e_{\BB(\A(X))}: T_{\BB(\A(X))}^2 \Rightarrow \id_{\BB(\A(X))},\]
 \item $T_{C,D}(\phi) = e_{\BB(\A(X))}(C)\circ T_{\BB(\A(X))}(\phi)$, for all chain maps $\phi:T_{\BB(\A(X))}(C)\to D$,
 \item $e_{\BB(\A(X))}(T_{\BB(\A(X))}(C))\circ T_{\BB(\A(X))}(e_{\BB(\A(X))}(C)) = \id_{T_{\BB(\A(X))}(C)}$, for all $C\in \BB(\A(X))$.
\end{enumerate}
\end{prop}
\begin{proof}
This is an immediate consequence of Proposition \ref{prop:eTTeFormulaFromSwitchNatIso} applied to $\A=\A(X)$ which may be applied since $T_{D,C}=T_{C,D}^{-1}$, for all $C$,$D\in \BB(\A(X))$. The proof of Proposition \ref{prop:eTTeFormulaFromSwitchNatIso} is categorical and does not depend on the codomain $\BB(\mathrm{Ab})$ of $-\otimes_\A-$. Therefore Proposition \ref{prop:eTTeFormulaFromSwitchNatIso} still holds for $\A=\A(X)$ using the refined tensor products of Definition \ref{defn:tensoroverAofX}.
\end{proof}

The rest of the section is devoted to proving that $e_{\BB(\A(X))}(C):T_{\BB(\A(X))}^2(C)\simeq C$, for all $C\in\BB(\A(X))$. By Proposition \ref{prop:chcont} it is sufficient to show that
\[e_{\BB(\A(X))}(C)_{\sigma,\sigma}:T_{\BB(\A(X))}^2(C)(\sigma)\to C(\sigma)\]is a chain equivalence in $\A$, for all $\sigma\in X$.

\begin{prop}\label{prop:decomposesIntoProjAndeBB}
Let $e_{\BB(\A(X))}: T_{\BB(\A(X))}^2 \Rightarrow \id_{\BB(\A(X))}$ be as defined in Proposition \ref{prop:defnOfeAX}. Then, for all $C\in \BB(\A(X))$ and all $\sigma\in X$, the map \[e_{\BB(\A(X))}(C)_{\sigma,\sigma}: T_{\BB(\A(X))}^2(C)(\sigma)\to C(\sigma)\]is the composition of a signed projection map
\[ T_{\BB(\A(X))}^2(C)(\sigma) \twoheadrightarrow T_{\BB}^2(C(\sigma))\]and 
\[e_\BB(C(\sigma)): T_{\BB}^2(C(\sigma)) \to C(\sigma).\]
\end{prop}
\begin{proof}
Applying $T_{C,D}= \Phi_{D,C}^{-1}\circ \mathrm{sh}(T_{(\sI_{X,\A})_\BB(C),(\sI_{X,\A})_\BB(D)})\circ \Phi_{C,D}$ in the case where $D=T_{\BB(\A(X))}(C)$ to 
\[\phi = \id_{T_{\BB(\A(X))}(C)}\in (C \otimes_{\A(X)} T_{\BB(\A(X))}(C))_0\]the components
\[e_{\BB(\A(X))}(C)(\sigma)_{p,q}^r:T_{\A(X)}(T_{\BB(\A(X))}(C)_p(\sigma))_q \to (\sI_{X,\A})_\BB(C)(\sigma)_r\]
of $e_{\BB(\A(X))}(C)=T_{C,T_{\BB(\A(X))}(C)}(\phi)$ can be calculated.

Using the signs for $\Phi_{C,D}$, $\Phi_{D,C}$ of Proposition \ref{prop:identifytensorprodfunctors} together with $(\ref{eqn:TBphiComponents1pt5})$ and Proposition \ref{prop:functorcatdualityisos} the components $e_{\BB(\A(X))}(C)(\sigma)_{p,q}^r$ are computed to be
\begin{equation}\label{eqn:cptsOfeBSigmaSigma}
(-1)^{|\sigma|(q-|\sigma|)}(-1)^{\frac{1}{2}|\sigma|(|\sigma|-1)}e_\A((\sI_{X,\A})_\BB(C)(\sigma)_r)_{q-|\sigma|,q-|\sigma|}\circ T_\A(\phi(\sigma)_{r,q}^p)_{q-|\sigma|} 
\end{equation}
for all $p-q+r=0$ where $\phi(\sigma)_{r,q}^p$ is the inclusion map
\[T_{\BB(\A(X))}(C_r)_q(\sigma)\hookrightarrow (\sI_{X,\A})_\BB(T_{\BB(\A(X))}(C)_p)(\sigma).\]
Examining $(\ref{eqn:cptsOfeBSigmaSigma})$ and using additivity of $e_\A$ it follows that
\[ e_{\BB(\A(X))}(C)_{\sigma,\sigma} = \pm e_\BB(C(\sigma))\circ \mathrm{proj.}: T_{\BB(\A(X))}^2(C)(\sigma) \to T_\BB^2(C(\sigma))\to C(\sigma)\]
as required.
\end{proof}
Consequently, since $e_\BB(C(\sigma)): T_{\BB}^2(C(\sigma)) \to C(\sigma)$ is a chain equivalence in $\A$, for all $\sigma\in X$, it is now sufficient to prove that the projection map 
\[ T_{\BB(\A(X))}^2(C)(\sigma) \twoheadrightarrow T_{\BB}^2(C(\sigma))\]of Proposition \ref{prop:decomposesIntoProjAndeBB} is a chain equivalence in $\A$. The signed projection map is a chain equivalence if and only if the unsigned projection map is; we prove this but must first make the following definitions.


\begin{defn}\label{defn:TheDsigmaFunctors}
For all $\sigma\in X$, define $\mathcal{D}_*^\sigma:X\to \BB(\A(\Z))$ in $\BB(\A(\Z))^*[X]$ by
\begin{align}
 \mathcal{D}^\sigma_*(\tau) &= C_*([\tau:\sigma];\Z),\notag\\
 \mathcal{D}^\sigma_*(\rho\to\tau) &= \mathrm{restriction}: C_*([\rho:\sigma];\Z)\to C_*([\tau:\sigma];\Z).\notag
\end{align}
 Similarly define $\mathcal{D}_\sigma^{-*}:X\to \BB(\A(\Z))$ in $\BB(\A(\Z))_*[X]$ by
\begin{align}
 \mathcal{D}_\sigma^{-*}(\tau) &= C^{-*}([\sigma:\tau] ;\Z),\notag\\
 \mathcal{D}_\sigma^{-*}(\rho\to\tau) &= \mathrm{restriction}: C^{-*}([\sigma:\tau];\Z)\to C^{-*}([\sigma:\rho];\Z).\notag
\end{align}
\end{defn}

\begin{defn}\label{defn:suspendedZoverAballFunctors}
For any $C\in\BB(\A)$, define the functor $C_\sigma:X\to \BB(\A)$ by
\begin{align}
 C_\sigma(\tau) &= \delta_{\tau\sigma} C(\sigma),\notag\\
  C_\sigma(\rho\to\tau) &= \delta_{\rho\sigma} \delta_{\tau\sigma} \id_{C(\sigma)},\notag
\end{align}
where $\delta_{\tau\sigma}$ and $\delta_{\rho\sigma}$ are Kronecker $\delta$ functions. Note that $C_\sigma$ is in both $\BB(\A)^*[X]$ and $\BB(\A)_*[X]$ as it is supported on a single ball $\sigma\in X$.
\end{defn}

\begin{prop}\label{prop:DsigmanatualisotoZSuspended}
For all $\sigma\in X$, there are natural equivalences of functors 
\[e:\mathcal{D}_*^\sigma \Rightarrow (\Sigma^{|\sigma|}\Z)_\sigma,\quad e^*:\mathcal{D}_\sigma^{-*} \Rightarrow (\Sigma^{-|\sigma|}\Z)_\sigma\]in $\BB(\A(\Z))^*[X]$ and $\BB(\A(\Z))_*[X]$ respectively.
\end{prop}
\begin{proof}
By Proposition \ref{prop:cellChainsOfRhoColonSigmaCtble} there are chain contractions \[P_\tau: \mathcal{D}_*^\sigma(\tau) \simeq 0,\quad Q_\tau : \mathcal{D}_\sigma^{-*}(\tau)\simeq 0,\]for all $\tau\neq \sigma$ and 
\[\mathcal{D}_*^\sigma(\sigma) = \Sigma^{|\sigma|}\Z,\quad \mathcal{D}_\sigma^{-*}(\sigma) = \Sigma^{-|\sigma|}\Z.\]Define the natural transformations $e$ and $e^*$ by
\begin{align}
e(\tau)   &= \delta_{\tau\sigma}\id_{\Sigma^{|\sigma|}\Z},\notag\\ 
e^*(\tau) &= \delta_{\tau\sigma}\id_{\Sigma^{-|\sigma|}\Z}.\notag
\end{align}
Checking these are natural transformations is straightforward and since
\begin{align}
e(\tau): \mathcal{D}_*^\sigma(\tau)&\simeq (\Sigma^{|\sigma|}\Z)_\sigma(\tau),\notag\\
e^*(\tau): \mathcal{D}_\sigma^{-*}(\tau)&\simeq (\Sigma^{-|\sigma|}\Z)_\sigma(\tau)\notag
\end{align}
are chain equivalences for all $\tau\in X$ we see that $e$ and $e^*$ are natural equivalences.
\end{proof}

\begin{prop}\label{prop:projMapIsChainEquiv}
The projection maps 
\[ T_{\BB(\A(X))}^2(C)(\sigma)\twoheadrightarrow T_{\BB}^2(C(\sigma))\]
of Proposition \ref{prop:decomposesIntoProjAndeBB} are chain equivalences.
\end{prop}
\begin{proof}
First observe that there are chain isomorphisms
\begin{align}
 T_{\BB(\A^*(X))}(C)|_\sigma &\stackrel{\cong}{\longrightarrow} \mathcal{D}_*^{\sigma}\otimes \mathcal{T}_\BB(C)\notag\\
 T_{\BB(\A_*(X))}(C)|_{\mathrm{st}(\sigma)} &\stackrel{\cong}{\longrightarrow} \mathcal{D}^{-*}_{\sigma}\otimes \mathcal{T}_\BB(C)\notag
\end{align}
in $\mathbb{G}_X(\A)$ given by redistributing summands on the left to different balls and redistributing morphisms to go between the new locations of the summands. For example,
\[ (T_{\BB(\A^*(X))}(C)|_\sigma)_n(\tau) = \sum_{\rho\leqslant\tau\leqslant\sigma}T_{\BB}(C(\rho))_{n-|\tau|}\]whereas
\begin{align}
 (\mathcal{D}_*^\sigma\otimes\mathcal{T}_\BB(C))_n(\rho) &= (C_*(\mathrm{st}(\rho)\cap\sigma;\Z)\otimes_\Z T_\BB(C(\rho)))_n\notag\\
 &= \sum_{\rho\leqslant\tau\leqslant\sigma}C_{|\tau|}(\tau,\partial\tau;\Z)\otimes_\Z T_\BB(C(\rho))_{n-|\tau|}\notag\\
 &= \sum_{\rho\leqslant\tau\leqslant\sigma}T_\BB(C(\rho))_{n-|\tau|}.\notag
\end{align}
In this case the redistribution isomorphism takes the summand $T_{\BB}(C(\rho))_{n-|\tau|}$ in $(T_{\BB(\A^*(X))}(C)|_\sigma)_n$ associated to the ball $\tau$ and moves it to lie over $\rho$, which is where it is in $(\mathcal{D}_*^\sigma\otimes\mathcal{T}_\BB(C))_n$.

The natural equivalences of Proposition \ref{prop:DsigmanatualisotoZSuspended} induce chain equivalences
\begin{align}
 e\otimes \id_{\mathcal{T}_\BB(C)}: \mathcal{D}_*^\sigma\otimes \mathcal{T}_\BB(C) &\stackrel{\simeq}{\longrightarrow} (\Sigma^{|\sigma|}\Z)_\sigma\otimes \mathcal{T}_\BB(C)\notag\\
 e^*\otimes \id_{\mathcal{T}_\BB(C)}: \mathcal{D}^{-*}_\sigma\otimes \mathcal{T}_\BB(C) &\stackrel{\simeq}{\longrightarrow} (\Sigma^{-|\sigma|}\Z)_\sigma\otimes \mathcal{T}_\BB(C)\notag
\end{align}
in $\A(X)$. Further, since the chain equivalences are the identity map over $\sigma\in X$ and map to zero elsewhere they are projections onto the components supported on $\sigma$. The chain homotopy inverses are correspondingly inclusions of the components supported on $\sigma$.

Assembling the above chain isomorphisms and chain equivalences gives chain equivalences in $\BB(\A)$:
\begin{align}
\mathrm{Ass}(T_{\BB(\A^*(X))}(C)|_\sigma) &\stackrel{\cong}{\longrightarrow} \mathrm{Ass}(\mathcal{D}_*^\sigma\otimes \mathcal{T}_\BB(C)) \notag\\
&\stackrel{\simeq}{\longrightarrow} \mathrm{Ass}((\Sigma^{|\sigma|}\Z)_\sigma\otimes \mathcal{T}_\BB(C)) = \Sigma^{|\sigma|}T_\BB(C(\sigma))\notag\\
\mathrm{Ass}(T_{\BB(\A_*(X))}(C)|_\sigma) &\stackrel{\cong}{\longrightarrow} \mathrm{Ass}(\mathcal{D}^{-*}_\sigma\otimes \mathcal{T}_\BB(C)) \notag\\
&\stackrel{\simeq}{\longrightarrow} \mathrm{Ass}((\Sigma^{-|\sigma|}\Z)_\sigma\otimes \mathcal{T}_\BB(C)) = \Sigma^{-|\sigma|}T_\BB(C(\sigma))\notag
\end{align}
which are also projections onto the component that used to be supported on $\sigma$ before assembly. Note that the chain homotopy inverses are the inclusions
\begin{align}
\Sigma^{|\sigma|}T_\BB(C(\sigma)) \hookrightarrow \mathrm{Ass}(T_{\BB(\A^*(X))}(C)|_\sigma) \notag\\ 
\Sigma^{-|\sigma|}T_\BB(C(\sigma)) \hookrightarrow \mathrm{Ass}(T_{\BB(\A_*(X))}(C)|_{\mathrm{st}(\sigma)}). \notag
\end{align}

Applying $\Sigma^{|\sigma|}T_\BB(-)$ and $\Sigma^{-|\sigma|}T_\BB(-)$ to these inclusions give projection maps
\begin{align}
T_{\BB(\A^*(X))}^2(C)(\sigma) = \Sigma^{|\sigma|}T_\BB(T_{\BB(\A^*(X))}(C)|_\sigma) &\stackrel{\simeq}{\longrightarrow} \Sigma^{|\sigma|}T_\BB(\Sigma^{|\sigma|}T_\BB(C(\sigma)))\notag\\
T_{\BB(\A_*(X))}^2(C)(\sigma) = \Sigma^{-|\sigma|}T_\BB(T_{\BB(\A_*(X))}(C)|_{\mathrm{st}(\sigma)}) &\stackrel{\simeq}{\longrightarrow} \Sigma^{-|\sigma|}T_\BB(\Sigma^{-|\sigma|}T_\BB(C(\sigma)))\notag
\end{align}
which are also chain equivalences as $T_\BB$ sends chain equivalences to chain equivalences by Proposition \ref{prop:TBSendsChainEquivsToChainEquivs}.

For a contravariant additive functor $\sG: \BB(\A)\to \BB(\A)$ it can be easily checked that there are isomorphisms of functors
\[\Sigma^{n}\circ \sG \cong \sG\circ \Sigma^{-n},\]for all $n\in \Z$. Consequently we have isomorphisms
\begin{align}
 \Sigma^{|\sigma|}T_\BB(\Sigma^{|\sigma|}T_\BB(C(\sigma))) &\cong \Sigma^{|\sigma|}\Sigma^{-|\sigma|}T_\BB^2(C(\sigma)) =T_\BB^2(C(\sigma))\notag\\
 \Sigma^{-|\sigma|}T_\BB(\Sigma^{-|\sigma|}T_\BB(C(\sigma))) &\cong \Sigma^{-|\sigma|}\Sigma^{|\sigma|}T_\BB^2(C(\sigma)) =T_\BB^2(C(\sigma)).\notag
\end{align}

Composing the chain equivalences of all the previous steps shows that the projection maps 
\[ T_{\BB(\A(X))}^2(C)(\sigma)\to T_{\BB}^2(C(\sigma))\]
are chain equivalences as required.
\end{proof}

\begin{ex}\label{ex:redistIsoEtcForCSupportedOnSigma}
Suppose $C\in\BB^*(\A(X))$ is such that $C(\tau)\neq 0$ if and only if $\tau=\rho$. Then by Example \ref{ex:TAXCWhenCSupportedOnSigma}
\[ T_{\BB^*(\A(X))}(C)(\sigma) = \brcc{C_*(\sigma;\partial\sigma;\Z)\otimes_\Z T_\BB(C(\rho)),}{\sigma\geqslant\rho,}{0,}{\mathrm{otherwise.}}\]
The boundary maps are such that 
\begin{equation}\label{eqn:AssOne}
\mathrm{Ass}(T_{\BB^*(\A(X))}(C)) \cong C_*(\mathrm{st}(\rho);\Z)\otimes_\Z T_\BB(C(\rho)). 
\end{equation}
Restricting $T_{\BB^*(\A(X))}(C)$ to $\sigma$, for any $\sigma\geqslant\rho$, the isomorphism $(\ref{eqn:AssOne})$ restricts to
\begin{equation}\label{eqn:AssTwo}
\mathrm{Ass}(T_{\BB^*(\A(X))}(C)|_\sigma) \cong C_*([\rho:\sigma];\Z)\otimes_\Z T_\BB(C(\rho)) = \mathrm{Ass}(\mathcal{D}^\sigma_*\otimes \mathcal{T}_\BB(C)) 
\end{equation}
where the equality in $(\ref{eqn:AssTwo})$ follows from the fact that $(\mathcal{D}^\sigma_*\otimes \mathcal{T}_\BB(C))(\tau)$ is non-zero only for $\tau=\rho$ so that $(\mathcal{D}^\sigma_*\otimes \mathcal{T}_\BB(C))(\rho) = \mathrm{Ass}(\mathcal{D}^\sigma_*\otimes \mathcal{T}_\BB(C))$.

The isomorphism $(\ref{eqn:AssTwo})$ is the total assembly of the redistribution isomorphism which maps
\[C_*(\tau,\partial\tau;\Z)\otimes_\Z T_\BB(C(\rho)) = T_{\BB^*(\A(X))}(C)(\tau)\]for all $\rho\leqslant\tau\leqslant\sigma$ from the ball $\tau$ to the summand \[C_*(\tau,\partial\tau;\Z)\otimes_\Z T_\BB(C(\rho)) = C_*(\tau,\partial\tau;\Z)\otimes_\Z \mathcal{T}_\BB(C)(\rho)\]of \[C_*([\rho:\sigma];\Z)\otimes_\Z \mathcal{T}_\BB(C)(\rho) = (\mathcal{D}^\sigma_*\otimes \mathcal{T}_\BB(C))(\rho)\]which is associated to the ball $\rho$.

Next 
\begin{align}
 T_{\BB^*(\A(X))}^2(C)(\sigma) &= C_*(\sigma,\partial\sigma;\Z)\otimes_\Z T_\BB(\mathrm{Ass}(T_{\BB^*(\A(X))}(C)|_\sigma))\notag\\
 &\cong \Sigma^{|\sigma|}T_\BB(\mathrm{Ass}(\mathcal{D}^\sigma_*\otimes \mathcal{T}_\BB(C)))\notag\\
 &\simeq \Sigma^{|\sigma|}T_\BB(\mathrm{Ass}((\Sigma^{|\sigma|}\Z)_\sigma\otimes \mathcal{T}_\BB(C)))\notag\\
 &= \brcc{\Sigma^{|\sigma|}T_\BB(\Sigma^{|\sigma|}T_\BB(C(\rho)))),}{\sigma=\rho,}{0,}{\sigma\neq\rho.}\notag
\end{align}
The last equality is due to $(\Sigma^{|\sigma|}\Z)_\sigma$ being $0$ except over $\sigma$ and $\mathcal{T}_\BB(C)$ being $0$ except over $\rho$. Thus 
$(\Sigma^{|\sigma|}\Z)_\sigma\otimes \mathcal{T}_\BB(C)$ is $0$ except over $\sigma$ if $\sigma=\rho$ in which case it is equal to $\Sigma^{|\sigma|}T_\BB(C(\rho))$. Hence we have
\[ T_{\BB^*(\A(X))}^2(C)(\rho) \simeq \Sigma^{|\sigma|}T_\BB(\Sigma^{|\sigma|}T_\BB(C(\rho))))\cong T_\BB^2(C(\rho))\]and 
\[ T_{\BB^*(\A(X))}^2(C)(\sigma) \simeq 0 = T_\BB^2(C(\sigma))\]for $\sigma\neq\rho$. The $C\in \BB_*(\A(X))$ case is similar.
\end{ex}

\section{$L$-theory} \label{sec:L-theory}

In this section we quickly review the $L$-groups, the $L$-spaces and the $L$-spectra of additive categories with chain duality. The main sources are \cite{Ranicki(1992)} and \cite{Weiss(1992)}, see also \cite{Kuehl-Macko-Mole(2012)}. The symbol $W$ denotes the standard $\ZZ[\ZZ/2]$-resolution of $\ZZ$, $\Delta^{k}$ is the standard $k$-simplex and $\AA$ is an additive category with chain duality as in the above sources.

\begin{definition}
	\cite[Definition 1.6]{Ranicki(1992)}
	\begin{enumerate}
		\item An $n$-dimensional {\it symmetric chain complex} over $\AA$ is a pair $(C,\varphi)$ where $C$ is a chain complex in $\BB (\AA)$ and $\varphi$ is an $n$-dimensional cycle in $W^{\%} (C) = \Hom_{\ZZ[\ZZ/2]} (W,C \otimes_{\AA} C)$. It is called {\it Poincar\'e} if $\varphi_{0} \co \Sigma^{n}TC \ra C$ is a chain equivalence.
		\item An $n$-dimensional {\it quadratic chain complex} over $\AA$ is a pair $(C,\psi)$ where $C$ is a chain complex in $\BB (\AA)$ and $\psi$ is an $n$-dimensional cycle in $W_{\%} (C) = W \otimes_{\ZZ[\ZZ/2]} (C \otimes_{\AA} C)$. It is called {\it Poincar\'e} if $(1+T) \psi_{0} \co \Sigma^{n}TC \ra C$ is a chain equivalence.
	\end{enumerate}
\end{definition}

\begin{definition} \label{defn:L-groups}
	\cite[Definition 1.8]{Ranicki(1992)}
	\begin{enumerate}
		\item The $n$-dimensional symmetric $L$-group $L^{n} (\AA)$ is defined to be the cobordism group of $n$-dimensional symmetric Poincar\'e complexes $(C,\varphi)$.
		\item The $n$-dimensional quadratic $L$-group $L_{n} (\AA)$ is defined to be the cobordism group of $n$-dimensional quadratic Poincar\'e complexes $(C,\psi)$.
	\end{enumerate}
\end{definition}

\begin{definition} \label{defn:L-spaces}
	\cite[Definition 13.2]{Ranicki(1992)}
	\begin{enumerate}
		\item The $n$-th symmetric $L$-space $\bL^{n} (\AA)$ is defined to be the Kan $\Delta$-set whose $k$-simplices are $(n+k)$-dimensional symmetric Poincar\'e complexes $(C,\varphi)$ in $\AA^{\ast}(\Delta^{k})$.
		\item The $n$-th quadratic $L$-space $\bL_{n} (\AA)$ is defined to be the Kan $\Delta$-set whose $k$-simplices are $(n+k)$-dimensional quadratic Poincar\'e complexes $(C,\psi)$ in $\AA^{\ast}(\Delta^{k})$.
	\end{enumerate}
\end{definition}

We have by \cite[Proposition 13.4]{Ranicki(1992)} that
\[
\pi_{k} \bL^{n} (\AA) = L^{n+k} (\AA) \quad \textup{and} \quad \pi_{k} \bL_{n} (\AA) = L_{n+k} (\AA).
\]

\begin{definition} \label{defn:L-spectra}
	\cite[Definition 13.5]{Ranicki(1992)}
	\begin{enumerate}
		\item The symmetric $L$-spectrum $\bL^{\bullet} (\AA)$ is defined to be the $\Omega$-spectrum of Kan $\Delta$-sets whose $n$-th space is $\bL^{n} (\AA)$.
		\item The quadratic $L$-spectrum $\bL_{\bullet} (\AA)$ is defined to be the $\Omega$-spectrum of Kan $\Delta$-sets whose $n$-th space is $\bL_{n} (\AA)$.
	\end{enumerate}
\end{definition}

In the next section we review a generalisation of all these notions so that an algebraic bordism category can be taken as an input. 

\section{Homology theory} \label{sec:homology-theory}

In this section we show that the collection of functors
\[
X \mapsto L_{n} (\AA_{\ast}(X)) = \pi_{n} (\bL_{\bullet} (\AA_{\ast}(X)))
\]
defines a homology theory on the category $\strcellcat$ of structured cell complexes. Note that this functor factors through the forgetful functor $\strcellcat \ra \ballcat \ra \regCWcplxcat$ and so we only need to show that this assignment is a homology theory on the category $\regCWcplxcat$.

We closely follow \cite{Weiss(1992)}, where the same was proved with the category of finite $\Delta$-sets and inclusions as the source. For the sake of notation and completeness we reproduce the proofs here as well, noting that they still work in this more general setting almost word for word.

We need to prove homotopy invariance and excision. The existence of long exact sequences is automatic since the functors above factor through spectra. A use is made of localisation sequences in algebraic surgery which was formalised by Ranicki using the language of algebraic bordism categories \cite[Chapter 3]{Ranicki(1992)}.

\begin{definition} \cite[Definitions 3.2,3.1]{Ranicki(1992)}
	An {\it algebraic bordism category} $\Lambda = (\AA,\BB,\CC)$ is an additive category $\AA$ with chain duality $(T,e)$ together with a pair $(\BB,\CC \subseteq \BB)$ of closed subcategories of $\BB (\AA)$ such that for any object $B \in \BB$ the cone $\sC (\id_{B})$ is an object in $\CC$ and $e (B) \co T^{2} (B) \ra B$ is a $\CC$-equivalence.
	A subcategory $\CC \subseteq \BB (\AA)$ is {\it closed} if it is a full additive subcategory such that the mapping cone of any chain map in $\CC$ is in $\CC$.
\end{definition}

\begin{definition} \cite[Definition 3.4]{Ranicki(1992)}
	Let $\Lambda = (\AA,\BB,\CC)$ be an algebraic bordism category.
	\begin{enumerate}
		\item The $n$-dimensional symmetric $L$-group $L^{n} (\Lambda)$ is defined to be the cobordism group of $n$-dimensional symmetric complexes $(C,\varphi)$ such that $C \in \BB$ and $\del C = \Sigma^{-1} \sC (\varphi_{0}) \in \CC$.
		\item The $n$-dimensional quadratic $L$-group $L_{n} (\Lambda)$ is defined to be the cobordism group of $n$-dimensional quadratic complexes $(C,\psi)$ such that $C \in \BB$ and $\del C = \Sigma^{-1} \sC ((1+T) \psi_{0}) \in \CC$.
	\end{enumerate}
\end{definition}

In other words the $L$-group $L_{n} (\Lambda)$ is a cobordism group of quadratic chain complexes in $\BB$ which are Poincar\'e modulo $\CC$. Similarly to before we can define $L$-spaces and $L$-spectra which will have the Kan property, so their homotopy groups will be the $L$-groups.

	Let $\Lambda = (\AA,\BB,\CC)$ be an algebraic bordism category and let $(X,A)$ be a pair of regular $CW$-complexes. In this situation one can define several algebraic bordism categories with the underlying additive category with chain duality $\AA_{\ast} (X)$. First denote by $\BB_{\ast} (X) = \BB (\AA_{\ast} (X))$ and by $\CC_{\ast} (X) \subseteq \BB_{\ast} (X)$ the subcategory of chain complexes in $\AA_{\ast}(X)$ which are contractible over all cells in $X$ and by $\CC_{A} (X) \subseteq \BB_{\ast} (X)$ the subcategory of chain complexes in $\AA_{\ast}(X)$ which are contractible over cells outside $A$. Note that we obviously have $\CC_{\ast} (X) \subseteq \CC_{A} (X)$. Then we denote $\Lambda_{\ast} (X) = (\AA_{\ast} (X),\BB_{\ast} (X),\CC_{\ast} (X))$ and $\Lambda_{A} (X) = (\AA_{\ast} (X),\BB_{\ast} (X),\CC_{A} (X))$ and $\Lambda^{A}_{\ast} (X) = (\AA_{\ast} (X),\CC_{A} (X),\CC_{\ast} (X))$.

\begin{remark}
	In the interest of brevity we sometimes shorten the notation. For example, if the category $\CC$ consists of all contractible chain complexes in $\BB = \BB(\AA)$ we just write $L_{n} (\AA)$ which is consistent with the previous section. See also \cite[Example 3.17]{Ranicki(1992)}
\end{remark}

Given an additive category $\AA$ with chain duality $(T,e)$ and a triple of closed subcategories $\DD \subseteq \CC \subseteq \BB \subseteq \BB (\AA)$ we obtain a homotopy fibration sequence of spectra \cite[Proposition 13.11]{Ranicki(1992)}
\begin{equation} \label{eqn:htpy-fib-seq-of-alg-bord-cast}
	\bL_{\bullet} (\AA,\CC,\DD) \ra \bL_{\bullet} (\AA,\BB,\DD) \ra \bL_{\bullet} (\AA,\BB,\CC).
\end{equation}

We note that this allows us to describe the relative term of a map of $L$-spectra using cobordisms of a single object rather than cobordisms of pairs, see \cite[Proposition 3.9]{Ranicki(1992)}. Let $(X,A)$ be a pair of regular $CW$-complexes. Since we have a triple of subcategories $\CC_{\ast} (X) \subseteq \CC_{\AA} (X) \subseteq \BB_{\ast} (X)$ this yields a homotopy fibration sequence
\begin{equation} \label{eqn:loc-seq-for-pair}
	\bL_{\bullet} (\Lambda^{A}_{\ast} (X)) \ra \bL_{\bullet} (\Lambda_{\ast} (X)) \ra \bL_{\bullet} (\Lambda_{A} (X)).
\end{equation}

The inclusion $A \hookrightarrow X$ induces a homotopy equivalence of $\bL_{\bullet} (\AA_{\ast}(A))$ with the first term by a homotopy invariance argument which is identical with the proof of \cite[Proposition 4.7]{Ranicki(1992)}. Hence a posteriori the third term will give a description of $L$-homology for the pair $(X,A)$.

To prove excision means to show that a pushout square of ball complexes induces a homotopy pushout square of the associated $L$-theory spectra. In order to analyse this square it is convenient to understand the cofibre of the induced map of the inclusion of a $(k-1)$-skeleton $X^{(k-1)}$ into the $k$-skeleton $X^{(k)}$ of a given regular $CW$-complex $X$ better via the localisation sequence above.

\begin{lemma} \label{lem:cofiber-induced-map-inclusion-k-1-skeleton-to-k-skeleton} \textup{\cite[Lemma 3.1]{Weiss(1992)}}
	The cofibre of the inclusion map
	\[
	\bL_{\bullet} (\AA_{\ast} (X^{(k-1)})) \ra 	\bL_{\bullet} (\AA_{\ast} (X^{(k)}))
	\]
	is homotopy equivalent to
	\[
	\vee_{\sigma \in X[k]} \Sigma^{k} \bL_{\bullet} (\AA))
	\]
	where $X[k]$ denotes the set of $k$-cells of $X$.
\end{lemma}

\begin{proof}
	Abbreviate $\CC_{k-1} (X)= \CC_{X^{(k-1)}} (X)$ and $\Lambda_{k-1} (X)= \Lambda_{X^{(k-1)}} (X)$ and also $\Lambda^{k-1}_{\ast} (X)= \Lambda^{X^{(k-1)}}_{\ast} (X)$. As noted above it follows from \cite[Proposition 4.7 (ii)]{Ranicki(1992)} that we have a homotopy equivalence
	\begin{equation} \label{eqn:inclusion-k-1-into-E}
	\bL_{\bullet} (\Lambda_{\ast} (X^{(k-1)})) \ra \bL_{\bullet} (\Lambda^{k-1}_{\ast} (X^{(k)})).	
	\end{equation}
	Consider now the homotopy fibration sequence
	\begin{equation} \label{eqn:loc-seq-for-inclusion}
		\bL_{\bullet} (\Lambda^{k-1}_{\ast} (X^{(k)})) \ra \bL_{\bullet} (\Lambda_{\ast} (X^{(k)})) \ra \bL_{\bullet} (\Lambda_{k-1} (X^{(k)})).
	\end{equation}
	The evaluation map on $k$-cells produces a functor
	\[
	\AA_{\ast} (X^{(k)}) \ra \prod_{\sigma \in X[k]} \AA \quad ; \quad M \ra (M(\sigma))_{\sigma \in X[k]}.
	\]
	This is a functor of additive categories with chain duality up to a dimension shift and therefore we obtain for all $q \in \ZZ$ a map
	\[
	\ev_{k} \co \bL_{q} (\AA_{\ast}(X^{(k)})) \ra \bL_{q-k} (\prod_{\sigma \in X[k]} \AA).
	\]
	This map factors through
	\begin{equation} \label{eqn:ev_k}
	\ev_{k} \co \bL_{q} (\Lambda_{k-1} (X^{(k)})) \ra \bL_{q-k} (\prod_{\sigma \in X[k]} \AA).
	\end{equation}
	Now both the source and the target of $\ev_{k}$ are Kan $\Delta$-sets. By an argument similar to the proof of the Kan property one can show that the map $\ev_{k}$ is a Kan fibration. The fiber over $0$ is
	\[
	\bL_{q} (\AA_{\ast} (X^{(k-1)}),\BB_{\ast}(X^{(k-1)}),\BB_{\ast}(X^{(k-1)})) \simeq \ast
	\]
	and hence the map \eqref{eqn:ev_k} is a homotopy equivalence. Finally note that the inclusions of the factors induce another homotopy equivalence
	\begin{equation} \label{eqn:wedge-versus-product}
	\vee_{\sigma \in X[k]} \Sigma^{k} \bL_{\bullet} (\AA) \ra \bL_{q-k} (\prod_{\sigma \in X[k]} \AA).
	\end{equation}
	Putting together \eqref{eqn:inclusion-k-1-into-E}, \eqref{eqn:loc-seq-for-inclusion}, \eqref{eqn:ev_k} and \eqref{eqn:wedge-versus-product} yields the desired result.
\end{proof}

For homotopy invariance the same spectral sequence argument as in \cite[Corollary 3.2]{Weiss(1992)} can be used.

Finally we can reproduce the excision.

\begin{proposition} \label{prop:excision} \textup{\cite[Corollary 3.3]{Weiss(1992)}}
	The functor
	\[
	 X \mapsto \bL_{\bullet} (\AA_{\ast} (X))
	\]
	is excisive. That is, if $X_{1}$ and $X_{2}$ are subcomplexes of a finite regular $CW$-complex $X$ with the intersection $X_{0} = X_{1} \cap X_{2}$ then the square
	\[
	\xymatrix{
	\bL_{\bullet} (\AA_{\ast} (X_{0})) \ar[r] \ar[d] & \bL_{\bullet} (\AA_{\ast} (X_{1})) \ar[d] \\
	\bL_{\bullet} (\AA_{\ast} (X_{2})) \ar[r] & \bL_{\bullet} (\AA_{\ast} (X))
	}
	\]
	is a homotopy pushout square.
\end{proposition}

\begin{proof}
	For a finite regular $CW$-complex $Y$ denote $G_{k} (Y) = \bL_{\bullet} (\AA_{\ast} (Y^{(k)}))$. By Lemma~\ref{lem:cofiber-induced-map-inclusion-k-1-skeleton-to-k-skeleton} the functor
	\[
	Y \ra G_{k} (Y) / G_{k-1} (Y)
	\]
	sends the square under consideration to a homotopy pushout square. By induction the same is true for the functor $G_{k} (-)$. If $k \geq \dim (X)$ we obtain the desired result.
\end{proof}

We conclude the section by unravelling what it means to have an element in $L$-homology of a ball complex $X$ and of a pair of ball complexes $(X,A)$. It is in fact very similar to what we have when $X$ is a simplicial complex as explained in \cite[Example 12.9]{Ranicki(1992)}. Hence an element $[(C,\psi)]$ in $L_{n} (\AA_{\ast} (X))$ is represented by a compatible collection of $(n-|\sigma|)$-dimensional quadratic Poincar\'e chain $(m-|\sigma|)$-ads
\begin{equation} \label{eqn:cycle-in-L-hlgy}
\sigma \mapsto (C(\sigma),\psi(\sigma)),	
\end{equation}
that means $(C(\sigma),\psi(\sigma)) \in \bL_{n-m}^{(m-|\sigma|)} (\AA)$ for some $m \in \ZZ$.

For a pair $(X,A)$ we obtain the same except we allow that $(C(\sigma),\psi(\sigma))$ is not necessarily Poincar\'e if $\sigma \in A$, see \eqref{eqn:loc-seq-for-pair}.

\begin{remark} \label{rem:connective-versions}
	For the sake of clarity we did not introduce connective versions of the algebraic bordism categories as in \cite[Chapter 15]{Ranicki(1992)}. However, the proofs would work just as well in those cases.
\end{remark}

\begin{remark} \label{rem:notation-Lambda}
	In the introduction we also used the notation $\Lambda^{c}_{\ast} (X)$ and $\Lambda^{\ast} (X)$ for certain algebraic bordism categories following partially \cite{Ranicki(1992)}. Let us explain here that $\Lambda^{c}_{\ast} (X)$ means the same as $\Lambda_{\ast} (X)$ together with the assumption that the chain complexes are globally contractible (that means contractible after the assembly \cite[Chapter 9]{Ranicki(1992)}). The category $\Lambda^{\ast} (X)$ is defined just like $\Lambda_{\ast} (X)$ except that the underlying additive category with chain duality is the category $\AA^{\ast} (X)$. For the sake of completeness we mention the related category $\Lambda (X)$ from \cite{Ranicki(1992)} which consists of all chain complexes in $\AA_{\ast} (X)$ which are globally Poincar\'e. Its $L$-theory is not a homology theory.
\end{remark} 

\section{Signatures} \label{sec:signatures}

In this section we review how to obtain elements of various $L$-groups from geometric situations. This is a straightforward generalisation from simplicial complexes to structured cell complexes of what was done in \cite{Ranicki(1992)} and \cite{Kuehl-Macko-Mole(2012)}. The main idea is to make maps transverse to dual cells. This can be done since for a structured cell complex $X$ each dual cell $D(\sigma,X)$ has a trivial normal PL-bundle with the fibre given by the ball $\sigma$ itself.

Let us first recall signatures over group rings. Let $F$ be a $k$-dimensional manifold with a reference map $r_{F} \co F \ra L$ to a structured cell complex $L$. The symmetric signature
\begin{equation} \label{eqn:sym-signature-over-group-ring}
	\ssign_{\pi_{1} (L)} (F) \in L^{k} (\ZZ \pi_{1} (L))
\end{equation}
is represented by a symmetric chain complex
\[
(C,\varphi) \quad \textup{such that} \quad C = C(F) \quad \textup{and} \quad \varphi = r_{F}^{\%}\varphi_{F} [F]
\]
where $\varphi_{F}$ is the symmetric construction of \cite[Section 1]{Ranicki-II-(1980)} which is natural on the chain level. There exists a relative version \cite[Section 6]{Ranicki-II-(1980)} which is also natural on the chain level.

Let $(f,b) \co M \ra X$ be a degree one normal map from an $n$-dimensional manifold $M$ to an $n$-dimensional geometric Poincar\'e complex $X$ with a reference map $r_{X} \co X \ra K$ to a structured cell complex $K$. Let $U \co \Sigma^{p} X_{+} \ra \Sigma^{p} M_{+}$ be an associated Umkehr map. The quadratic signature
\begin{equation} \label{eqn:quad-signature-over-group-ring}
 \qsign_{\pi_{1} (K)} (f,b) \in L_{n} (\ZZ \pi_{1} (K))
\end{equation}
is represented by the quadratic chain complex
\[
(C,\psi) \quad \textup{such that} \quad C = \sC (f^{!}) \quad \textup{and} \quad \psi = (r_{X})_{\%} \psi_{U} [X]
\]
where $f^{!}$ denotes the algebraic Umkehr map associated to $U$ and $\psi_{U}$ is the quadratic construction of \cite[Section 1]{Ranicki-I-(1980)}. Again there exists a relative version \cite[Section 6]{Ranicki-II-(1980)}.

\begin{remark} \label{rem:quad-construction-is-dissection-friendly}
We note that while the quadratic construction $\psi_{U}$ in \cite[Section 1]{Ranicki-I-(1980)} and its relative version in \cite[Section 6]{Ranicki-II-(1980)} are described on the chain level, they are not known to be natural on the chain level (only in homology), see \cite[Chapter I, page 30]{Ranicki(1981)}. On the other hand we observe that a substitute property exists which is often sufficient for arguments on the chain level. As noted in the above sources, the only problem preventing us from the naturality property is that there is no natural inverse to the suspension chain homotopy equivalence $C(X) \ra \Sigma^{-p} C (\Sigma^{p} (X))$. Suppose that $X$ is ``dissected over'' $K$ as in \cite[Example 9.12]{Ranicki(1992)}. The suspension map respects the dissection and for each simplex $\sigma$ it restricts to the suspension map for $X[\sigma]$ and hence they individually have chain homotopy inverses. But now the argument from \cite[Proposition 4.7]{Ranicki(1992)} tells us that there exists an inverse for $X$ which respects the dissection. 	
\end{remark}

In the following discussion we will encounter several algebraic bordism categories as in the previous section. The underlying additive category $\AA$ will be the category $\ZZ$ of abelian groups and hence the notation $\ZZ_{\ast} (X)$ stands for the category obtained as in Section \ref{sec:chain-duality-on-cats-over-ball-cplxs-saf} from $\ZZ$.

Let $F$ be a $k$-dimensional manifold with a reference map $r_{F} \co F \ra L$ which is a homotopy equivalence to a structured cell complex $L$. (Such $r_{F}$ always exists for example using the fact that $F$ is an ENR. If $F$ is a triangulated manifold we can take $L$ to be the underlying simplicial complex of a triangulation with $r_{F}$ the identity. Also the map does not have to be a homotopy equivalence, but for applications we are mainly interested in that case, so we assume it.) Make $r_{F}$ transverse to dual cells so that each $F (\sigma) = r_{F}^{-1} (D(\sigma,L))$ is an $(k-|\sigma|)$-dimensional manifold with boundary. A choice of the fundamental class $[F] \in C_{k} (F)$ projects to a choice of the fundamental class for each $C_{k-|\sigma|} (F(\sigma),\del F(\sigma))$. By naturality of the symmetric construction on the chain level we obtain relative symmetric signatures for all $\sigma \in L$ which fit together to yield the symmetric signature over $L$
\begin{equation} \label{eqn:sym-sign-manifold-over-X}
	\ssign_{L} (F) \in L^{k} (\ZZ_{\ast} (L))
\end{equation}
represented by a symmetric chain complex over $L$
\[
(C,\varphi) \quad \textup{such that} \quad C(\sigma) = C (F(\sigma),\del F(\sigma)) \quad \textup{for} \quad \sigma \in L
\]
and $\varphi (\sigma)$ is the relative symmetric structure for $(F(\sigma),\del F(\sigma))$ obtained as in \cite[Section 6]{Ranicki-II-(1980)}. See \cite[Example 9.12]{Ranicki(1992)}, \cite[Definition 8.11]{Kuehl-Macko-Mole(2012)} for details.

Let $(F,\del F)$ be a $k$-dimensional manifold with boundary and with a reference map $r_{(F,\del F)} \co (F,\del F) \ra (L,\del L)$ which is a homotopy equivalence of pairs to a pair of structured cell complexes $(L,\del L)$. We obtain
\begin{equation} \label{eqn:sym-sign-manifold-over-X-rel-case}
	\ssign_{(L,\del L)} (F,\del F) \in L^{k} (\Lambda_{\del L} (L))
\end{equation}
again represented by a symmetric chain complex over $L$
\[
(C,\varphi) \quad \textup{such that} \quad C(\sigma) = C (D(\sigma),\del D(\sigma)) \quad \textup{for} \quad \sigma \in L,
\]
and $\varphi (\sigma)$ is the relative symmetric structure for $(F(\sigma),\del F(\sigma))$ obtained as in \cite[Section 6]{Ranicki-II-(1980)}. We note however, that the chain complex $(C,\varphi)$ here is in general only Poincar\'e over balls outside $\del L$.

Let $(f,b) \co M \ra X$ be a degree one normal map between two $n$-dimensional manifolds with a reference map $r_{X} \co X \ra K$ which is a homotopy equivalence to a structured cell complex $K$ and such that both $r_{X}$ and $r_{X} \circ f$ are transverse to the dual cells of $K$, so that we have a compatible collection of degree one normal maps $(f(\sigma),b(\sigma)) \co (M(\sigma),\del M(\sigma)) \ra (X(\sigma),\del X(\sigma))$. Let $U \co \Sigma^{p} X_{+} \ra \Sigma^{p} M_{+}$ be an associated Umkehr map. It projects to Umkehr maps for all $\sigma \in K$. Using the substitute for the naturality of the quadratic construction on the chain level described in Remark \ref{rem:quad-construction-is-dissection-friendly} we obtain the relative quadratic signatures for each $\sigma \in K$ which fit together to yield a quadratic signature over $K$
\[
\qsign_{K} (f,b) \in L_{n} (\ZZ_{\ast} (K))
\]
represented by the quadratic chain complex over $K$
\[
(C,\psi) \quad \textup{such that} \quad C(\sigma) = \sC (f^{!}(\sigma),\del f^{!}(\sigma)) \quad \textup{for} \quad \sigma \in K.
\]
Here $f (\sigma) = f|_{(r_{X} \circ f)^{-1} (D(\sigma,L))}$, the map $f^{!} (\sigma)$ denotes the algebraic Umkehr map obtained from the suitable projection of $U$ and $\psi (\sigma)$ is the resulting relative quadratic structure associated to the degree one normal map $(f(\sigma),b(\sigma))$ obtained as in \cite[Section 6]{Ranicki-II-(1980)}. For details see \cite[Example 9.13]{Ranicki(1992)}, \cite[Definition 8.14]{Kuehl-Macko-Mole(2012)}.

Let $(f,b) \co (M,\del M) \ra (X,\del X)$ be a degree one normal map between two $n$-dimensional manifolds with boundary with a reference map $r_{(X,\del X)} \co (X,\del X) \ra (K,\del K)$ which is a homotopy equivalence of pairs to a pair of structured cell complexes and such that both $r_{(X,\del X)}$ and $r_{(X,\del X)} \circ f$ are transverse to the dual cells of $K$. We obtain the relative quadratic signatures for each $\sigma \in K$ which fit together to yield a quadratic signature over $(K,\del K)$
\[
\qsign_{(K,\del K)} (f,b) \in L_{n} (\Lambda_{\del K} (K))
\]
represented by the quadratic chain complex over $K$
\[
(C,\psi) \quad \textup{such that} \quad C(\sigma) = \sC (f^{!}(\sigma),\del f^{!}(\sigma)) \quad \textup{for} \quad \sigma \in K.
\]
Again here $f (\sigma) = f|_{(r_{X} \circ f)^{-1} (D(\sigma,L))}$, the map $f^{!} (\sigma)$ denotes the algebraic Umkehr map obtained from the suitable projection of $U$ and $\psi (\sigma)$ is the resulting relative quadratic structure associated to the degree one normal map $(f(\sigma),b(\sigma))$ obtained as in \cite[Section 6]{Ranicki-II-(1980)}. We note however, that the chain complex $(C,\psi)$ here is in general only Poincar\'e over balls outside $\del K$.

\begin{remark} \label{rem:global-version}
	We note that without much effort the signatures we presented here can also be generalised to the case when $X$ is a geometric Poincar\'e complex and not necessarily a manifold. However, one should observe that the main difference in this situation is that when considering signatures over $X$ one does not obtain elements of the $L$-theory of the category $\Lambda_{\ast} (X)$, only in the $L$-theory of the category $\Lambda (X)$, see Remark \ref{rem:notation-Lambda}, that means that one obtains quadratic chain complexes dissected over $X$, but they will not necessarily be locally Poincar\'e, they might only be Poincar\'e globally (essentially since $X$ might not be locally Poincar\'e it is only globally Poincar\'e by definition.) See \cite[Section 14]{Kuehl-Macko-Mole(2012)} for more details about signatures in these situations.
\end{remark}


\section{Products} \label{sec:products}


For symmetric and quadratic chain complexes over rings we have products and for symmetric and quadratic signatures over group rings we have product formulae as follows.

Recall from \cite[Section 8]{Ranicki-I-(1980)} that for rings $R$ and $S$ with involution we have natural products
\begin{align}
	\begin{split}
	\label{eqn:products-on-L-groups-of-rings}
	\blank \otimes \blank & \co L^{k} (R) \otimes L^{n} (S) \ra L^{n+k} (R \otimes S) \\
	\blank \otimes \blank & \co L^{k} (R) \otimes L_{n} (S) \ra L_{n+k} (R \otimes S)		
	\end{split}
\end{align}
such that in both cases the underlying chain complex of the product is the usual tensor product of the underlying chain complexes
\begin{align}
	\begin{split}
	\label{eqn:products-on-L-groups-of-rings-formulas}
 (C,\varphi_{C}) \otimes (D,\varphi_{D}) & = (C \otimes D,\varphi_{C} \otimes \varphi_{D}) \\
 (C,\varphi_{C}) \otimes (D,\psi_{D}) & = (C \otimes D,\varphi_{C} \otimes \psi_{D})
	\end{split}
\end{align}
and where $\varphi_{C} \otimes \varphi_{D}$ and $\varphi_{C} \otimes \psi_{D}$ are defined via a diagonal approximation of $W$ as in \cite[page 174]{Ranicki-I-(1980)}. Note that these definitions are again on the chain level.

Recall from \cite[Section 8]{Ranicki-II-(1980)} that for $r_{F} \co F \ra B \pi$, $r_{F'} \co F' \ra B \pi'$ and $(f,b) \co M \ra X$ with $r_{X} \co X \ra B \pi'$ these products satisfy product formulae
\begin{align}
	\begin{split}
		\label{eqn:product-formulas-for-signatures-over-group-rings}
			\ssign_{\ZZ[\pi \times \pi']} (F \times F') & = \ssign_{\ZZ[\pi]} (F) \otimes \ssign_{\ZZ[\pi']} (F') \\
			\qsign_{\ZZ[\pi \times \pi']} (\id_{F} \times f, \id_{\nu_F} \times b) & = \ssign_{\ZZ[\pi]} (F) \otimes \qsign_{\ZZ[\pi']} (f,b).
		\end{split}
\end{align}
In the symmetric case the proof uses that the acyclic models method induces a natural chain homotopy equivalence $C(F \times F') \ra C(F) \otimes C(F')$ inverse to the cross product. The cross product produces a fundamental class cycle of $F \times F'$ from the fundamental class cycles of $F$ and $F'$ and the same argument as the one which shows the Cartan formula for Steenrod squares identifies the two symmetric structures. In the quadratic case the proof uses in addition that if a stable map $U \co \Sigma^{p} X_{+} \ra \Sigma^{p} M_{+}$ is a geometric Umkehr map for $(f,b)$ then $\id \wedge U \co \Sigma^{p} (F \times X)_{+} \ra \Sigma^{p} (F \times M)_{+}$ is a geometric Umkehr map for $(id \times f,\id \times b)$ and that the symmetric construction commutes naturally with suspensions. As before note that these arguments work in the relative case and on the chain level in the symmetric case and in the quadratic case we can use Remark \ref{rem:quad-construction-is-dissection-friendly}.

The relative versions are straightforward. We emphasise that the above discussion says that the products are constructed naturally on the chain level in the symmetric case. In the quadratic case almost the same is true, use the substitute for naturality discussed in the previous section which allows us to say that we get the multiplicativity already on the chain level. The fact that the formulae in \cite{Ranicki-II-(1980)} are stated in terms of $L$-groups is caused by the observation that the fundamental classes are only well-defined up to homology on the chain level.

Our aim in this section is to refine these products to products
\begin{align}
	\begin{split} \label{eqn:products-on-L-groups-over-X}
	\blank \otimes \blank & \co L^{k} (\ZZ_{\ast} (L)) \otimes L^{n} (\ZZ_{\ast} (K)) \ra L^{n+k} (\ZZ_{\ast} (L \times K)) \\
	\blank \otimes \blank & \co L^{k} (\ZZ_{\ast} (L)) \otimes L_{n} (\ZZ_{\ast} (K)) \ra L_{n+k} (\ZZ_{\ast} (L \times K))
	\end{split}
\end{align}
which for $r_{F} \co F \ra L$, $r_{F'} \co F' \ra K$ and $(f,b) \co M \ra X$ with $r_{X} \co X \ra K$  satisfy
\begin{align}
	\begin{split} \label{eqn:product-formula-for-signatures-over-X}
	\ssign_{L \times K} (F \times F') & = \ssign_{L} (F) \otimes \ssign_{K} (F'), \\
	\qsign_{L \times K} (\id_{F} \times f, \id_{\nu_{F}} \times b) & = \ssign_{L} (F) \otimes \qsign_{K} (f,b).
	\end{split}
\end{align}

We also obtain relative versions which we discuss in more detail later.

Recall that an element in $L^{k} (\ZZ_{\ast} (L))$ is represented by a $k$-dimensional symmetric chain complex $(C,\varphi_{C})$ in $\ZZ_{\ast} (L)$. This means that it is an assignment $\sigma \mapsto (C(\sigma),\varphi(\sigma))$ where the value is an appropriate $(k-|\sigma|)$-dimensional symmetric Poincar\'e ad. Let similarly $(D,\varphi_{D})$ represent an element in $L^{n} (\ZZ_{\ast} (K))$ and $(D,\psi_{D})$ represent an element in $L_{n} (\ZZ_{\ast} (K))$. Define the products \eqref{eqn:products-on-L-groups-over-X} by the formulae
\begin{align}
	\begin{split}
		((C,\varphi_{C}) , (D,\varphi_{D})) & \mapsto (\sigma \times \tau \mapsto (C(\sigma) \otimes D(\tau),\varphi_{C} (\sigma) \otimes \varphi_D (\tau))) \\
		((C,\varphi_{C}) , (D,\psi_{D})) & \mapsto (\sigma \times \tau \mapsto (C(\sigma) \otimes D(\tau),\varphi_{C} (\sigma) \otimes \psi_D (\tau))).
	\end{split}
\end{align}
where the products $\otimes$ on the right hand sides are the chain level products in \eqref{eqn:products-on-L-groups-of-rings-formulas}. The products here are well defined, since products of Poincar\'e ads are Poincar\'e ads of correct dimensions, see \eqref{eqn:cycle-in-L-hlgy}.

With all the work done so far, the construction is straightforward due to the fact that the ball complex structure on a product of ball complexes is given by products of balls and that the dual cells are products of dual cells and that the above product formulae actually come from product formulae on the chain level.

Let $r_{F} \co F \ra L$ and $r_{F'} \co F' \ra K$ be transverse to dual cells and consider $\ssign_{L} (F)$ and $\ssign_{K} (F')$. These are represented by symmetric chain complexes over $L$ and $K$ respectively with underlying chain complexes such that for $\sigma \in L$ we have that $C (\sigma) = C(D(\sigma,L),\del D (\sigma,L))$ and for $\tau \in K$ we have that $D (\tau) = C(D(\tau,K),\del D (\tau,K))$.

Now consider $r_{F} \times r_{F'} \co F \times F' \ra L \times K$. This map is already transverse to the dual cells $D(\sigma \times \tau, L \times K) = D(\sigma, L) \times D(\tau, K)$. Moreover we have that $r_{F} \times r_{F'}$ restricted to $(r_{F} \times r_{F'})^{-1} (D(\sigma \times \tau, L \times K))$ equals
\[
r_{F}| \times r_{F'}| \co F(\sigma) \times F'(\tau) \ra D(\sigma,L) \times D(\tau,K).
\]
Hence from the multiplicativity of the relative products on the chain level we obtain the desired first formula in \eqref{eqn:product-formula-for-signatures-over-X}.

For the quadratic case let $(f,b) \co M \ra X$ with $r_{X} \co X \ra K$ and $r_{F} \co F \ra L$ be transverse to the dual cells. Then both $r_{F} \times r_{X}$ and $(r_{F} \times r_{X}) \circ f$ are already transverse to the dual cells $D(\sigma \times \tau,L \times K) = D(\sigma, L) \times D(\tau, K)$ and the restriction maps are the product maps
\[
(\id_{F(\sigma)} \times f(\tau),\id_{\nu_{F(\sigma)}} \times b(\tau)) \co F (\sigma) \times M(\tau) \ra F(\sigma) \times X (\tau).
\]
Hence from the multiplicativity of the relative products on the chain level we obtain the desired second formula in \eqref{eqn:product-formula-for-signatures-over-X}.

There are many relative versions of the products and product formulae above. Note that they are obtained exactly as the absolute versions, except we observe that the symmetric and quadratic complexes which are involved are only required to be Poincar\'e outside the boundaries. For simplicity we only discuss one of them which is also of interest for applications. Let $(F,\del F)$ be a $k$-dimensional manifold with boundary and with a reference map $r_{(F,\del F)} \co (F,\del F) \ra (L,\del L)$ which is a homotopy equivalence of pairs to a pair of structured cell complexes $(L,\del L)$. In the previous section we constructed
\[
\ssign_{L,\del L} (F,\del F) \in L^{k} (\Lambda_{\del L} (L)).
\]
We have the relative product
\begin{equation} \label{eqn:products-on-L-groups-over-X-relative-case}
\blank \otimes \blank \co L^{k} (\Lambda_{\del L} (L)) \otimes L_{n} (\ZZ_{\ast} (K)) \ra L_{n+k} (\Lambda_{\del L \times K} (L \times K))	
\end{equation}
and again from the multiplicativity of the relative products on the chain level we obtain the product formula
\begin{equation} \label{eqn:product-formula-for-signatures-over-X-relative-case}
	\qsign_{(L \times K, \del L \times K)} (\id_{F} \times f, \id_{\nu_{F}} \times b) = \ssign_{(L,\del L)} (F,\del F) \otimes \qsign_{K} (f,b).
\end{equation}

Let us now consider the special case when $(F,\del F) = (D^{k},S^{k-1})$ and we take $r_{D^{k}}$ to be the identity and think of $D^{k} = [0,1]^{k}$ as a ball complex. Also recall that the $L$-groups over complexes are isomorphic to the homology groups with coefficients in the appropriate spectra. We observe that the product with
\[
\ssign_{(D^{k},S^{k-1})} (D^{k},S^{k-1}) \in L^{k} (\Lambda_{S^{k-1}} (D^{k})) \cong H_{k} (D^{k},S^{k-1};\bL^{\bullet})
\]
commutes with the suspension isomorphism
\[
H_{n} (K;\bL_{\bullet}) \ra H_{n+k} (D^{k} \times K,S^{k-1} \times K;\bL_{\bullet})
\]
as follows.

We observe from the definitions that the Poincar\'e duality isomorphism in symmetric $L$-theory sends
\begin{align*}
	L^{k} (\Lambda_{S^{k-1}} (D^{k})) \cong H_{k} (D^{k},S^{k-1};\bL^{\bullet}) & \cong H^{0} (D^{k};\bL^{\bullet}) \cong \pi_{0} \bL^{\bullet} \\
	\ssign_{(D^{k},S^{k-1})} (D^{k},S^{k-1}) & \mapsto 1
\end{align*}
and hence the symmetric signature $\ssign_{(D^{k},S^{k-1})} (D^{k},S^{k-1})$ is in fact the fundamental class of $(D^{k},S^{k-1})$ in the homology groups with coefficients in the symmetric $L$-theory spectrum $\bL^{\bullet}$.
Now the general theory of products in ring spectra and module spectra over ring spectra as in \cite{Adams(1974)} tells us that the product with the fundamental class of $(D^{k},S^{k-1})$ gives the suspension isomorphism.

\begin{remark}
	We note that without much effort the products and the product formulae we obtained here can be extended to a situation when $K$ and $L$ are geometric Poincar\'e complexes and not necessarily manifolds. The main difference is that, as noted above in Remark \ref{rem:global-version}, the signature lands in the $L$-theories of algebraic bordism categories of chain complexes which are only globally Poincar\'e. We leave the details for  the reader.	
\end{remark}

\small
\bibliography{tso-all}  
\bibliographystyle{alpha}

\end{document}